\input ./style/arxiv-general.cfg
\documentclass[aop,MSNbibl,dvips]{arximspdf}
\makeatletter
   \@ifpackageloaded{graphicx}{}{\usepackage{graphicx}}
\makeatother

\doi{10.1214/14-AOP949} 
\volume{43}
\issue{5}
\pubyear{2015}
\firstpage{2729}
\lastpage{2762}
\docsubty{FLA}

\makeatletter
\def\bigtriangleup{\Delta}

\newcommand{\rrvert}{\vert}

\newcommand{\rrVert}{\Vert}
\newcommand{\llvert}{\vert}
\newcommand{\llVert}{\Vert}
\def\shuffle{}
\newtheorem{theorem}{Theorem}[section]
\newtheorem{corollary}{Corollary}[section]
\newproclaim{definition}{Definition}[section]
\newproclaim{example}{Example}[section]
\newtheorem{lemma}{Lemma}[section]
\newproclaim{remark}{Remark}[section]
\makeatother

\begin{document}
\begin{frontmatter}

\title{Expected signature of Brownian motion up to the first
exit time from a bounded domain}
\runtitle{Expected signature of stopped Brownian motion}

\begin{aug}
\author[A]{\fnms{Terry}~\snm{Lyons}\corref{}\thanksref{T1}\ead[label=e1]{Terry.Lyons@maths.ox.ac.uk}}
\and
\author[A]{\fnms{Hao}~\snm{Ni}\thanksref{T2}\ead[label=e3]{ni@maths.ox.ac.uk}}
\thankstext{T1}{Supported by EPSRC (Grant Reference: EP/F029578/1)
and ERC Grant (Grant Agreement No. 291244 Esig).}
\thankstext{T2}{Supported by ERC Grant (Grant Agreement No. 291244 Esig).}
\runauthor{T. Lyons and H. Ni}
\affiliation{University of Oxford}

\address[A]{Oxford-Man Institute of Quantitative Finance\\
University of Oxford\\
Eagle House\\
Walton Well Road\\
OX2 6ED, Oxford\\
United Kingdom\\
\printead{e1}\\
\phantom{E-mail:\ }\printead*{e3}}
%
\end{aug}



\received{\smonth{5} \syear{2013}}
\revised{\smonth{6} \syear{2014}}

%
\begin{abstract}
The signature of a path provides a top down description of the path in
terms of its effects as a control [\textit{Differential Equations Driven by Rough Paths}
(2007) Springer]. The signature
transforms a path into a group-like element in the tensor algebra and
is an essential object in rough path theory. The expected signature of
a stochastic process plays a similar role to that played by the
characteristic function of a random variable. In [Chevyrev (2013)], it is
proved that under certain boundedness conditions, the expected value of
a random signature already determines the law of this random signature.
It becomes of great interest to be able to compute examples of expected
signatures and obtain the upper bounds for the decay rates of expected
signatures. For instance, the computation for Brownian motion on $
[ 0,1 ]$ leads to the ``cubature on Wiener space'' methodology
[Lyons and Victoir,
\textit{Proc. R. Soc. Lond. Ser. A Math. Phys. Eng. Sci.}
\textbf{460}
(2004)
169--198]. In this paper we fix a bounded domain
$\Gamma$ in a Euclidean space $E$ and study the expected signature of
a Brownian path starting at $z\in\Gamma$ and stopped at the first
exit time from $\Gamma$. We denote this tensor series valued function
by $\Phi_{\Gamma}(z)$ and focus on the case $E=\mathbb{R}^{d}$. We show
that $\Phi_{\Gamma}(z)$ satisfies an elliptic PDE system and a boundary
condition. The equations determining $\Phi_{\Gamma}$ can be recursively
solved; by an iterative application of Sobolev estimates we are able,
under certain smoothness and boundedness condition of the domain $\Gamma
$, to prove geometric bounds for the terms in $\Phi_{\Gamma}(z)$.
However, there is still a gap and we have not shown that $\Phi_{\Gamma
}(z)$ determines the law of the signature of this stopped Brownian
motion even if $\Gamma$ is a unit ball.
\end{abstract}

%

%
\end{frontmatter}

\section{Introduction}\label{sec1}
An essential notion in the theory of rough paths is the signature of a
path, which represents the path information in terms of its effects as
a control \cite{RoughPaths}. In \cite{Cubature}, Lyons and Victoir
developed a methodology called a \textit{cubature on the Wiener Space} to
numerically solve high dimensional SDEs and semi-elliptic PDEs, which
can be seen as an alternative to the Monte Carlo method. The
Lyons--Victoir method relies on the computation of the expected
signature of the Brownian motion up to an arbitrarily fixed time $T \in
\mathbb{R}^{+}$. Let $S(B_{[0, T]})$ denote the signature of the
Brownian motion up to time $T$. In \cite{Fawcett}, a closed-form
expression for the expected signature of the Brownian motion up to time
$T$, that is, the expectation of $S(B_{[0, T]})$, is given by the following:
%
\begin{equation}
\label{FawcettFormula} \mathbb{E} \bigl[S(B_{[0, T]}) \bigr] = \exp \Biggl(
\frac{T}{2} \sum_{i=1}^{d}e_{i}
\otimes e_{i} \Biggr),
\end{equation}
where $B$ is a standard $d$-dimensional Brownian motion with canonical basis
$(e_{1}, e_{2}, \ldots, e_{d})$ for $\mathbb{R}^{d}$.

In \cite{friz2012levy}, Fawcett's formula (\ref{FawcettFormula}) has
been extended to L\'evy processes up to fixed time. In this paper, we
consider the Brownian case when $T$ is a certain kind of the stopping
time as another extension of Fawcett's formula (\ref{FawcettFormula}).
More specifically, we consider the first exit time of the Brownian
motion $B$ from $\Gamma$, denoted by $\tau_{\Gamma}$, where $\Gamma$ is
a bounded domain in $E:=\mathbb{R}^{d}$. Denote the expected signature
of the Brownian motion $B$ starting at $z \in\Gamma$ up to $\tau
_{\Gamma}$ by $\Phi_{\Gamma}(z):= \mathbb{E}^{z}[S(B_{[0, \tau_{\Gamma
}]})]$. We prove that $\Phi_{\Gamma}$ satisfies an elliptic PDE taking
values in the tensor algebra space, with a boundary condition and an
initial condition given as follows:
\[
\cases{ %
\displaystyle\Delta \bigl(\Phi_{\Gamma}(z) \bigr)=-
\Biggl(\sum_{i=1}^{d}e_{i}\otimes
e_{i} \Biggr)\otimes \Phi _{\Gamma}(z)-2\sum
_{i=1}^{d}e_{i}\otimes\frac{\partial\Phi_{\Gamma
}(z)}{%
\partial z_{i}}, &
\quad $\mbox{$\forall z\in\Gamma$}$, \vspace*{2pt}
\cr
\displaystyle\lim_{t \uparrow\tau_{\Gamma}}
\Phi_{\Gamma}(B_{t})=\mathbf{1} \mbox{ a.s. } P^{z},
& \quad$\mbox{$\forall z\in\Gamma$}$ \vspace*{2pt}
\cr
\rho_{0} \bigl(
\Phi_{\Gamma}(z) \bigr)=1,\rho_{1} \bigl(\Phi_{\Gamma}(z)
\bigr)=0, & \quad $\mbox{$\forall z\in\bar{\Gamma}$},$}
\]
where $\rho_{n}$ is defined in Definition~\ref{rho_n_def}.

Using this PDE, we compute each term of the expected signature
recursively. In the case of the domain $\Gamma$ being a unit disk we
extend this result further, by demonstrating that the expected
signature has polynomial form with common factor $(1-|z|^{2})$. In
addition we derive a finite difference equation that characterizes the
expected signature of a simple $d$-dimensional random walk up to the
first exit time from a bounded domain.

In Section~\ref{GeoBoundsForPhi}, we study the decay rate of each term
in $\Phi_{\Gamma}$ and prove that $\Phi_{\Gamma}(z)$ is geometrically
bounded if $\Gamma$ is strongly Lipschitz and belongs to the class
$C^{m}$ where $m = \lfloor\frac{d}{2}\rfloor+1$ by using our PDE
approach in conjunction with Sobolev's theorem. This is motivated by a
series of studies on whether the law of a signature can be determined
by its expectation. The first relevant result was due to Fawcett~\cite{Fawcett}. Recently this result has been extended significantly in \cite{Ilya}, and a sufficient condition for the law of random signature
$S(X)$ to be uniquely determined by its expected value is that its
expected signature is compact-like in the sense of the second version
of the paper \cite{Ilya}, that is, for every $\delta>0$, there exists
a positive integer $N_{\delta}$, such that for every $n \geq N_{\delta}$,
\[
\bigl\Vert\rho_{n} \bigl(\mathbb{E} \bigl[S(X) \bigr] \bigr) \bigr\Vert<
\delta^{n}.
\]
However, the geometric boundedness of $\Phi_{\Gamma}$ does not imply
that $\Phi_{\Gamma}(z)$ is compact-like. It is still an open question
whether $\Phi_{\Gamma}(z)$ determines the law of the signature of the
Brownian motion up to $\tau_{\Gamma}$.

In \cite{NiHaoThesis}, a parabolic result of the expected signature of
a general time-homogenous It\^o diffusion up to fixed time $T$ is
obtained using the similar PDE approach. More specifically, let $X_{t}$
be a time-homogeneous $E$-valued It\^o diffusion satisfying the
following SDE:
\[
dX_{t} = \mu(X_{t})\,dt + V(X_{t})
\,dW_{t},
\]
where $W_{t}$ is a standard multi-dimensional Brownian motion. Let $\Phi
(T,x)$ denote the expected signature of $X$ from time $0$ to $T$. Under
some regularity assumptions of $\mu$, $V$ and $\Phi$, $\Phi$ satisfies
the following PDE:
\[
\cases{ %
\displaystyle \biggl( -\frac{\partial}{\partial T} + A \biggr)\Phi(T,x)
+ \sum^ {d} _{j = 1} \Biggl(\sum
^{d} _{j_{1} = 1} b_{j_{1}, j}(x) e_{j_{1}}
\Biggr) \otimes \frac{\partial\Phi
(T,x)}{\partial x_{j}} \vspace*{2pt}
\cr
\qquad\hspace*{-6pt} {}+
\displaystyle \Biggl( \sum^{d} _{j = 1}
\mu^{j}(x)e_{j} + \frac
{1}{2}\sum
^{d} _{j_{1} = 1} \sum^{d}
_{j_{2} = 1}b_{j_{1}, j_{2}}(x) e_{j_{1}} \otimes e_{j_{2}}
\Biggr) \otimes\Phi(T,x) = 0, \vspace*{2pt}
\cr
\Phi(0,x) = \mathbf{1},
\rho_{0} \bigl(\Phi(T,x) \bigr) = 1,}
\]
%
where $b(x) = V(x)V(x)^{T}$.

\section{Preliminary}
In this paper, we are mostly interested in describing probability
measures on paths. Despite our examples being quite specific, they can
be understood in a more wider context. For the sake of precision, we
start by
introducing some notation, making essential definitions and stating the
basic results we require. These can also be found in \cite{RoughPaths}. A
reader experienced in rough path theory might prefer to go directly to
Section~\ref{section2}.

\subsection{Tensors products}
Throughout the rest of the paper, fix $E=\mathbb{R}^{d}$ as the space in
which paths will take their values. Then $E$ has the basis $ \{
e_{1},\ldots,e_{d} \} $. Consider the successive tensor powers $%
E^{\otimes n}$ of $E$ (equipped with some tensor norm). If we think of the
elements $e_{i}$ as letters, then $E^{\otimes n}$ is spanned by the
words of
length $n$ in the letters $ \{ e_{1},\ldots,e_{d} \} $, and
can be
identified with the space of real homogeneous noncommuting polynomials of
degree $n$ in $d$ variables. We note that $E^{\otimes0}=\mathbb{R}$. In
order for our analysis to work it will be necessary to constrain the norms
we use when considering tensor products (the injective and projective norms
satisfy our constraints).

\begin{definition}
We say that the tensor powers of $E$ are endowed with an admissible
norm $|\cdot|$, if the following conditions hold:

\begin{longlist}[(1)]
\item[(1)] for each $n\geq1$, the symmetric group $S_{n}$ acts by
isometry on $E^{\otimes n}$, that is,
\[
|\sigma v|=|v|\qquad \forall v\in E^{\otimes n},\forall\sigma\in
S_{n};
\]
\item[(2)] the tensor product has norm $1$, that is, $\forall n,m\geq1$,
\[
|v\otimes w|\leq|v||w|\qquad\forall v\in E^{\otimes n},w\in E^{\otimes m}.
\]
\end{longlist}
\end{definition}

\subsection{The algebra of tensor series}

\begin{definition}
A formal $E$-tensor series is a sequence of tensors, denoted by $ (
a_{n}\in E^{\otimes n} ) _{n\in\mathbb{N}}$, which we write
$a= (
a_{0},a_{1},\ldots ) $. There are two binary operations on
$E$-tensor series, an addition $+$ and a product $\otimes$, which are defined
as follows. Let $\mathbf{a}=(a_{0},a_{1},\ldots)$ and $\mathbf{b}
=(b_{0},b_{1},\ldots)$ be two $E$-tensor series. Then we define
%
\begin{equation}
\mathbf{a}+\mathbf{b}=(a_{0}+b_{0},a_{1}+b_{1},
\ldots)
\end{equation}
and
%
\begin{equation}
\mathbf{a}\otimes\mathbf{b}=(c_{0},c_{1},\ldots),
\end{equation}
where for each $n\geq0$,
%
\begin{equation}
c_{n}=\sum_{k=0}^{n}a_{k}
\otimes b_{n-k}.
\end{equation}
The product $\mathbf{a}\otimes\mathbf{b}$ is also denoted by $\mathbf
{a}\mathbf{b}$. We use the notation $\mathbf{1}$ for the series $(1,0,\ldots)$,
and $\mathbf{0}$ for the series $(0,0,\ldots)$. If $\lambda\in\mathbb
{R}$, then we define $\lambda\mathbf{a}$ to be $(\lambda a_{0},\lambda
a_{1},\ldots)$.
\end{definition}

\begin{definition}
The space $T ( ( E )  ) $ is defined to be the vector
space of all formal $E$-tensors series.
\end{definition}

\begin{remark}
The space $T (  ( E )  ) $ with $+$ and $\otimes$
and the
action of $\mathbb{R }$ is an associative and unital algebra over
$\mathbb{R}
$. An element of $\mathbf{a}=(a_{0},a_{1},\ldots)$ of $T (  (
E )
 ) $ is invertible if and only if $a_{0}\neq0$. Its inverse is then
given by the series
%
\begin{equation}
\mathbf{a}^{-1}=\frac{1}{a_{0}}\sum_{n\geq0}
\biggl(\mathbf{1}-\frac{\mathbf
{a}}{%
a_{0}} \biggr)^{n},
\end{equation}
which is well defined because, for each given degree, only finitely many
terms of the sum produce nonzero tensors of this degree. In
particular, the
subset $\{\mathbf{a}\in T (  ( E )  ) |a_{0}=1\}$
forms a
group.
\end{remark}

\begin{definition}
The dilation operator denoted by $\delta_{\varepsilon}$ is a mapping
from $%
\mathbb{R}^{+}\times T (  ( E )  ) \rightarrow T (
 ( E )  ) $ defined by
\[
\delta_{\varepsilon}(\mathbf{a}) = \bigl(a_{0},\varepsilon
a_{1},\ldots,\varepsilon ^{n}a_{n},\ldots \bigr)
\qquad \forall\varepsilon\in\mathbb{R}^{+},\mathbf{a}\in T ( ( E )
).
\]
\end{definition}

\begin{definition}
Let $n\geq1$ be an integer. Let $B_{n}=\{\mathbf{a}%
=(a_{0},a_{1},\ldots)|a_{0}=\cdots=a_{n}=0\}$. The truncated tensor algebra $%
T^{(n)}(E)$ of order $n$ over $E$ is defined as the quotient algebra
%
\begin{equation}
T^{(n)}(E)=T ( ( E ) ) /B_{n}.
\end{equation}
The canonical homomorphism $T (  ( E )  )
\longrightarrow
T^{(n)}(E)$ is denoted by $\pi_{n}$.
\end{definition}

\subsection{Paths}
Paths also have algebraic properties:

\begin{definition}
Let $X\dvtx [r,s]\longrightarrow E$ and $Y\dvtx [s,t]\longrightarrow E$ be two
continuous paths. Their concatenation is the path $X\ast Y$ defined by
\[
(X\ast Y)_{u}=\cases{ %
X_{u}, &\quad $u\in [
r,s ],$ \vspace*{2pt}
\cr
X_{s}+Y_{u}-Y_{s}, &
\quad $u\in [ s,t ].$}
\]
\end{definition}

\begin{remark}
$\ast$ is an associative operation. Let $X\dvtx [r,s]\longrightarrow E$, $%
Y\dvtx [s,t]\longrightarrow E$ and $Z\dvtx [t,v]\longrightarrow E$ be
three continuous paths. Then
\[
(X\ast Y)\ast Z=X\ast(Y\ast Z).
\]
\end{remark}

\subsection{The signature of a path}

\begin{definition}
Let $J$ denote a compact interval. Let $X\dvtx J\longrightarrow E$ be a
path of
bounded variation or a rough path of finite $p$-variation such that the
following integration makes sense. The signature of $X$, denoted by $S(X_{J})
$, is an element $(1,X^{1},\ldots,X^{n},\ldots)$ of $T (  ( E )
 ) $ defined for each $n\geq1$ as follows:
\[
X^{n}=\mathop{\int\cdots\int}_{u_{1}<\cdots<u_{n}, u_{1},\ldots,u_{n}\in J}%
\,dX_{u_{1}} \otimes\cdots\otimes \,dX_{u_{n}}.
\]
The truncated signature of $X$ of order $n$ is denoted by $S^{n}(X)$,
that is, $%
S^{n}(X) = \pi_{n}(S(X_{J}))$, for every $n \in\mathbb{Z}$.
\end{definition}


\begin{example}
%
%
(1) If $X_{t}$ is a continuous path with finite $p$-variation,
where $1
\leq p<2$, then its signature can be defined in the sense of the Young
integral \cite{RoughPaths}. More generally, if $X_{t}$ is a $p$-rough
path ($%
p \geq1$), then the integrals will exist as a result of the general theory
of rough paths (Theorem~3.7, page 45, \cite{RoughPaths}).

(2) For a Brownian path, its signature can be defined in the
sense of the It\^{o} integral or the Stratonovich integral. There is a simple
rewriting rule
that allows one to go between them \cite{gaines1994algebra}. Almost all
Brownian paths, with their L\'evy area processes, are $p$-rough paths
for any $%
p>2$. With probability one, for all $(s,t)$ the Stratonovich signature
agrees with the canonical rough path signature (\cite{RoughPaths}, page 57).
%
\end{example}

Chen's identity is a fundamental theorem, which asserts that the signature
is a homomorphism between path space and rough path space.

\begin{theorem}\label{Chen}
Let $X\dvtx [r,s]\longrightarrow E$ and $Y\dvtx [s,t]\longrightarrow E$ be two
continuous paths with finite $p$-variation where $1 \leq p < 2$. Then
%
\begin{equation}
S(X\ast Y)=S(X)\otimes S(Y).
\end{equation}
\end{theorem}

The proof can be found in \cite{RoughPaths}.

\subsection{Real-valued functions on the signatures of paths}
We now introduce a special class of linear forms on $T (  (
E ) ) $; see \cite{RoughPaths}. Suppose $(e_{1}^{\ast
},\ldots,e_{i}^{\ast},\ldots)$ are elements of $E^{\ast}$. We can introduce
coordinate iterated integrals by setting
\[
X_{u}^{(i)}:= \bigl\langle e_{i}^{\ast},X_{u}
\bigr\rangle
\]
and rewriting $ \langle e_{i_{1}}^{\ast}\otimes\cdots\otimes
e_{i_{n}}^{\ast},S ( X )  \rangle$ as the scalar iterated
integral of coordinate projections
\[
\mathop{ \int\cdots\int}\limits
_{{ \mbox{$u_{1}
<\cdots<u_{n}\atop u_{1},\ldots,u_{n}\in J$}}}\,dX_{u_{1}}^{(i_{1})}\otimes
\cdots \otimes \,dX_{u_{n}}^{(i_{n})}.
\]
In this way, we realize $n$th degree coordinate iterated integrals as the
restrictions of linear functionals in $E^{\otimes n}$ to the space of
signatures of paths. If $(e_{1},\ldots,e_{d}) $ is a basis for a finite
dimensional space $E$, and $(e_{1}^{\ast},\ldots,e_{d}^{\ast})$ is a basis
for the dual space~$E^{\ast}$, we can write
\[
\mathbf{X}_{J}=\sum_{n=0}^{\infty}
\sum_{i_{1},\ldots, i_{n}\in\{
1,\ldots,
d\}} \biggl(\mathop{ \int\cdots\int}\limits
_{{ \mbox{$u_{1}
<\cdots<u_{n}\atop u_{1},\ldots,u_{n}\in J$}}}
\,dX_{u_{1}}^{(i_{1})} \otimes\cdots\otimes \,dX_{u_{n}}^{(i_{n})}
\biggr) e_{i_{1}} \otimes\cdots\otimes e_{i_{n}}.
\]

\begin{definition}
We define $T(E)\subset T (  ( E )  ) $ to be the tensor
series $\mathbf{a}=(a_{0},a_{1},\ldots)$ for which there exists $N$
depending on
$\mathbf{a}$ so that $a_{i}=0$ for all  $i>N$. In other words it is the space
of polynomials (instead of series) in elements of $E$.
\end{definition}

\begin{remark}
There is a natural inclusion
\[
T \bigl(E^{\ast} \bigr)\rightarrow T ( ( E ) )
^{\ast}.
\]
\end{remark}

\begin{remark}
For any $\mathbf{e}^{\ast}\in T ( E^{\ast} ) $, we denote by
$%
\mathcal{U}_{\mathbf{e}^{\ast}}$ the restriction of the linear map
$\langle
\mathbf{e}^{\ast},\cdot\rangle$ to the range of the signature. It is a
real-valued function on the set of the signature of paths. It is well known
that for any $\mathbf{e}^{\ast}$ and $\mathbf{f}^{\ast}$ in $T (
E^{\ast} ) $, the pointwise product $\mathcal{U}_{\mathbf
{e}^{\ast}}%
\mathcal{U}_{\mathbf{f}^{\ast}}$ equals $\mathcal{U}_{\mathbf{e}^{\ast
}%
\shuffle\mathbf{f}^{\ast}}$ for the shuffle product $\mathbf{e}^{\ast}
\shuffle\mathbf{f}^{\ast}\in T ( E^{\ast} ) $. In other
words any
polynomial in coordinate iterated integrals can be written as a linear
combination of iterated integrals \cite{RoughPaths}.
\end{remark}

\begin{remark}
The previous remark explains why, if we have the probability measure on the
signatures of paths, then the expected signature is a powerful piece of
information if it exists. The integral of any polynomial against this
measure can be calculated as follows. Use the shuffle product to identify
the linear functional on tensors that coincides with the polynomial on the
set of signatures. The polynomial and the linear functional have the same
integral, and this will be the contraction of the linear functional
with the
expectation of the measure (the expected signature). So the expected
signature determines the integral of the measure against any
polynomial. Of
course if the measure were compactly supported, then the polynomials are
dense in the continuous real-valued functions and the measure would be
completely determined by its expectation. This theorem has been
extended in \cite{Ilya} recently.
\end{remark}

\begin{remark}
Re-parameterizing a path inside the interval of definition does not change
its signature over the maximal interval. Translating a path does not change
its signature. We may define an equivalence relation between paths by asking
that they have the same signature.
\end{remark}

Returning to the finite dimensional example and the notation introduced
above, the linear forms $e_{I}^{\ast}$, as $I$ spans the set of finite
words in the dual letters $e_{1}^{\ast},\ldots,e_{d}^{\ast}$, form a basis
for $T(E^{\ast})$. For convenience, we fix two useful functions on
$T (
 ( E )  ) $: $\pi^{I}$ and $\rho_{n}$.

\begin{definition}
$\pi^{I}$ is defined by
\begin{eqnarray*}
\pi^{I}\dvtx T ( ( E ) ) &\rightarrow&\mathbb{R};
\\
\mathbf{a} &\mapsto& e_{I}^{\ast}(\mathbf{a}).
\end{eqnarray*}
\end{definition}

In particular,
\[
\pi^{I} \bigl(S(X) \bigr)=\mathcal{U}_{e_{I}^{\ast}} \bigl(S(X)
\bigr)= \mathop{ \int\cdots\int}\limits
_{{ \mbox{$u_{1}
<\cdots<u_{n}\atop u_{1},\ldots,u_{n}\in J$}}} \,dX_{u_{1}}^{(i_{1})}
\otimes\cdots \otimes \,dX_{u_{n}}^{(i_{n})},
\]
where $I=(i_{1},\ldots,i_{n})$. We call $\pi^{I}(S(X))$ the coordinate
signature of $X$ indexed by $I$.

\begin{definition}
\label{rho_n_def} $\rho_{n}$ is defined by
\begin{eqnarray*}
\rho_{n} \dvtx T ( ( E ) ) &\rightarrow& E^{\otimes n},
\\
(a_{0},a_{1},\ldots) &\mapsto&a_{n},
\end{eqnarray*}
where $a_{n}\in E^{\otimes n}$ and $n\in\mathbb{N}$.
\end{definition}

\subsection{The homogeneous Carnot--Caratheodory norm on the truncated
signature of a path}
In this subsection, we consider continuous paths of finite
$1$-variation and
look at their signatures truncated to order $n$. The range of this
signature map, a
subset of $T^{ ( n ) } ( E ) $, forms a group. We will
work with the Carnot--Caratheodory norm (CC norm) on this space. We
refer to
\cite{MultidimensionalStochasticProcessesAsRoughPaths} for a detailed
discussion of this norm but note that it is obvious from the definition that
the CC norm is invariant under rotation of paths. We start with the
definition of the space of continuous paths of finite $1$-variation.

\begin{definition}
Let $E$ be a Euclidean space endowed with Euclidean metric $\mathbf{d}$
and $x\dvtx [0,T]\rightarrow E$. For $0 \leq t\leq\bar{t} \leq T$, the
$1$-variation of
$x$ on $[t,\bar{t}]$ is defined as
\[
|x|_{1\mbox{-}\mathrm{var};[t,\bar{t}]} = \sup_{\mathcal
{D}=(t=u_{1}<u_{2}<\cdots<u_{n}=\bar{t}%
)} \sum
_{i=0}^{n-1} \mathbf{d}(x_{u_{i}},
x_{u_{i+1}} ).
\]
If $|x|_{1\mbox{-}\mathrm{var};[t,\bar{t}]} < \infty$, we say that $x$ is of bounded
variation or of finite $1$-variation on $[t,\bar{t}]$. The space of
continuous paths of finite $1$-variation on $[0, T]$ is denoted by $%
C^{1\mbox{-}\mathrm{var}}([0,T], E)$.
\end{definition}

\begin{definition}
The set of all signatures, truncated at order $N$, of bounded variation
paths is denoted by
\[
G^{N}(E):= \bigl\{S^{N}(x)_{0,1}\dvtx x\in
C^{1\mbox{-}\mathrm{var}} \bigl([0,1],E \bigr) \bigr\}.
\]
\end{definition}

The group $G^{N}(E)$ admits a number of equivalent metric structures.
Two will be important in this paper. The so-called {\it
Carnot--Caratheodory norm} is defined (and is finite) for every $g \in
G^{N}(E)$ as
\[
\llVert g\rrVert =\inf \bigl\{ \Vert \gamma\Vert _{1\mbox{-}\mathrm{var}}\dvtx
S^{ N } ( \gamma ) =g, \gamma\in C^{1\mbox{-}\mathrm{var}} \bigl([0,1], E \bigr)
\bigr\}.
\]
%
A second ``norm'' can be built directly out of the signature
\[
|\!|\!| g |\!|\!|:=\max_{i = 1,\ldots,N} \bigl|\rho_{i}(g
)\bigr|^{1/i}\qquad \forall g \in G^{N}(E).
\]
Both ``norms'' are homogeneous of degree 1 in scaling $\delta
_{\varepsilon}$. The group is finite dimensional. Therefore $\llVert
g\rrVert $ and $|\!|\!| g |\!|\!|$ are equivalent (page 11, \cite{{MultidimensionalStochasticProcessesAsRoughPaths}}).

\begin{definition}
The $p$-variation of a geometric rough path $\gamma$ defined on $ [
0,T %
 ] $, and written $\Vert S^{N}(\gamma)\Vert _{p\mbox{-}\mathrm{var};[0,T]}$, is defined on
paths in $G^{N }( E) $ to be
\[
\bigl\Vert S^{N }(\gamma) \bigr\Vert _{p\mbox{-}\mathrm{var};[0,T]}=\sup_{\mathcal{D}%
=(0=u_{1}<u_{2}<\cdots<u_{n}=T)}
\biggl( \sum_{\mathcal{D}} \bigl\llVert S^{N}(
\gamma)_{u_{i},u_{i+1}} \bigr\rrVert ^{p} \biggr) ^{1/p},
\]
where $\llVert \cdot \rrVert $ is the Carnot--Caratheodory norm.
\end{definition}

\begin{lemma}
\label{EquivalenceOf Homogeneous norms}
There is a constant $C$ depending on $d$ and $p$ such that for any path
$\gamma$ of bounded variation,\setcounter{footnote}{2}\footnote{$\llVert \cdot\rrVert $ and
$|\!|\!|\cdot|\!|\!|$ are different.}
\begin{eqnarray*}
\biggl\llvert \int_{0<u_{1}<\cdots<u_{N}<T}\,d\gamma_{u_{1}}\otimes
\cdots \otimes \,d\gamma_{u_{N}} \biggr\rrvert ^{1/N} &\leq& \bigl|\!
\bigl| \! \bigl| S^{N}(\gamma)\bigr |\! \bigr|\! \bigr| \leq C \bigl\llVert S^{N}(
\gamma) \bigr\rrVert
\\
&\leq& C \bigl\llVert S^{N}(\gamma) \bigr\rrVert _{p\mbox{-}\mathrm{var};[0,T]}.
\end{eqnarray*}
\end{lemma}

\subsection{Sobolev space}

\begin{definition}
Let $u$ be a locally integrable function in $\Gamma$ and $\alpha$ be a
multi-index. Then a locally integrable function $r_{\alpha} u$ such
that for every $g \in C_{c}^{\infty}(G)$,
\[
\int_{\Gamma} g(x) r_{\alpha}(x) \,dx =
(-1)^{\vert\alpha\vert} \int_{\Gamma} D^{\alpha} g(x) u(x)
\,dx,
\]
will be called the weak derivatives of $u$, and $r_{\alpha}$ is denoted
by $D^{\alpha} u$. By convention, $D^{\alpha}u = u$ if $\vert\alpha
\vert= 0$.
\end{definition}

Throughout the rest of discussion in this subsection, $\Gamma$ is a
subset of $E: = \mathbb{R}^{d}$.

\begin{definition}
The Sobolev space is defined to be the set of all $\mathbb{R}^{\tilde
{d}}$-valued functions $u \in L^{p}(\Gamma)$ such that for every
multi-index $\alpha$ with $\vert\alpha\vert\leq k$, the weak partial
derivative $D_{\alpha}u$ belongs to $L^{p}(\Gamma)$, that is,
\[
W^{k,p}(\Gamma) = \bigl\{ u \in L^p(\Gamma) \dvtx
D^{\alpha}u \in L^p(\Gamma)\ \forall|\alpha| \leq k \bigr\},
\]
where $k \in\mathbb{N}$, and $\Gamma$ is an open set.\vadjust{\goodbreak}

The Sobolev norm is defined as follows:\vspace*{-2pt}
\[
\Vert u \Vert_{W^{k,p}(\Gamma)} = \sum_{j=1}^{\tilde{d}}
\biggl(\sum_{\vert\alpha\vert\leq k} \int_{\Gamma}
\bigl\vert D_{\alpha} u^{j}(x) \bigr\vert^{p}\,dx
\biggr)^{1/p}.
\]
When $k=0$, $\Vert u \Vert_{W^{k,p}(\Gamma)}$ is also
$L^{p}(\Gamma)$-norm, that is,\vspace*{-1pt}
\[
\Vert u \Vert_{W^{k,p}(\Gamma)} = \Vert u \Vert _{L^{p}(\Gamma)}.
\]
\end{definition}

\begin{definition}
A domain $\Gamma\subset E$ is said to be strongly Lipschitz if and
only if $\Gamma$ is bounded and each point $x_{0}$ of $\partial\Gamma$
is in a neighborhood $\mathcal{R}$ which is the image under a rotation
and translation of axes of a domain $\{x:=(x_{1} ,\ldots, x_{d})\dvtx \vert
x_{d}'\vert< R, \vert x_{d}\vert< 2LR\}$ in which $x_{0}$ corresponds
to the origin, $\mathcal{R} \cap\partial\Gamma$ corresponds to the
locus $x_{d} = f(x_{d}')$ where $x_{d}' = (x_{1}, \ldots, x_{d-1})$, and
$f$ satisfies the Lipschitz condition with constant $L$ as well as\vspace*{-1pt}
\[
f(\mathcal{R} \cap\Gamma) = \bigl\{x\dvtx \bigl\vert x_{d}'
\bigr\vert<R, f \bigl(x_{d}' \bigr) < x_{d} < 2LR
\bigr\}.
\]
\end{definition}

Let us introduce one useful theorem, which states how the $L^{\infty}$
norm of $u$ can be controlled in terms of $L^{p}$ norm of the weak
derivatives of $u$ up to some degree under certain regularity condition
of a domain $\Gamma$; see Theorem~3.5.1 in~\cite{morrey2008multiple}.\vspace*{-2pt}

%
\begin{theorem}\label{SobEmbThml}
Suppose $u \in W^{m,p}(\Gamma)$ where $\Gamma$ is strongly Lipschitz in
$E$ and $m > d/p$. Then $u$ is continuous on $\Gamma$ and there is a
constant $C(d,m,p,\Gamma)$ such that
\[
\bigl\vert u(x) \bigr\vert\leq C \vert\Gamma\vert^{-1/p} \Biggl\{ \sum
_{j =
0}^{m-1} \frac{R^{j}}{j!} \bigl\Vert
\!\bigtriangledown^{j} u \!\bigr\Vert _{L^{p}(\Gamma)}+(m-v/p)^{-1}
\frac{R^{m}}{(m-1)!} \bigl\Vert\! \bigtriangledown^{m} u\! \bigr\Vert_{L^{p}(\Gamma)}
\Biggr\},
\]
where $R:=R(\Gamma)$ is the diameter of $\Gamma$ and $u$ can be a
vector valued function.\vspace*{-2pt}
\end{theorem}

%
\begin{lemma}\label{SobEmbThm}
Let $u\dvtx \Gamma\rightarrow\mathbb{R}^{\tilde{d}}$ where $\Gamma$ is a
strongly Lipschitz domain in $E$. Let $m=\lfloor d/2\rfloor+1$ and $u
\in W^{m,2}(\Gamma)$. Then $u$ is continuous on $\Gamma$ and there is a
constant $C:= C(d,\Gamma)$ such that\vspace*{-2pt}
\[
\bigl\vert u(x) \bigr\vert\leq C \vert\Gamma\vert^{-1/2} \max \biggl\{\max
^{m-1}_{j=0} \biggl( \frac{R^{j}}{j!} \biggr),
\frac
{R^{m}}{(m-1)!} \biggr\} \Vert u\Vert_{W^{m,2}(\Gamma)},
\]
where $R$ is the diameter of $\Gamma$.\vspace*{-2pt}
\end{lemma}

\begin{pf}
This is just a special case when $p = 2$, $m = \lfloor d/2\rfloor+1$ in
Theorem~\ref{SobEmbThml}. The proof is completed by noticing that\vspace*{-2pt}
\[
\sum_{j=0}^{m} \bigl\Vert\! \bigtriangledown^{j}
u \!\bigr\Vert_{2}^{0} \leq m \Vert u \Vert_{W^{m,2}(\Gamma)}.
\]\vspace*{-2pt}
\upqed\end{pf}

Then we will introduce an auxiliary result about boundary regularity
for solutions of linear elliptic equations \cite{jost1998partielle}.
Let $\mathbb{D}(z, r)$ denote an open $d$-dimensional ball centered at
$z$ with
radius $r$, and for simplicity let $\mathbb{D}$ denote the
$d$-dimensional open unit ball centered
at the origin. To state it precisely, we introduce the definition of a
domain of the class $C^{k}$.

\begin{definition}
An open and bounded set $\Gamma\subset E$ is of class $C^{k}(k = 0, 1,
\ldots, \infty)$ if for any $x_{0} \in\partial\Gamma$ there exist
$r>0$ and a bijective map $F\dvtx\break \mathbb{D}(x_{0},r) \subset\mathbb
{R}^{d}$ with the following properties:
\begin{longlist}[(1)]
\item[(1)] $F( \Gamma\cap\mathbb{D}(x_{0},r)) \subset\{(x_{1}, \ldots, x_{d})\dvtx x_{d} \geq0\}$.
\item[(2)] $F( \partial\Gamma\cap\mathbb{D}(x_{0},r)) \subset\{
(x_{1}, \ldots, x_{d})\dvtx x_{d} = 0\}$.
\item[(3)] $F$ and $F^{-1}$ are of class $C^{k}$.
\end{longlist}

%
\begin{remark}
This means that $\partial\Gamma$ is a $(d-1)$-dimensional $C^{k}$
submanifold of $\mathbb{R}^{d}$.
\end{remark}
\end{definition}

Let us consider the following class of elliptic differential equations:
\[
Mu := \sum_{i,j} \frac{\partial}{\partial x^{j}}
\biggl(a^{i,j}(x) \frac
{\partial}{\partial x^{i}}u(x) \biggr),
\]
where $a^{i,j}$ satisfies the ellipticity condition, namely that, there
exists some $\lambda>0$, with
\[
\sum_{i,j=1}^{d} a^{ij}(x)
\xi_{i}\xi_{j} \geq\lambda\vert\xi\vert ^{2}\qquad
\forall x \in\Gamma, \xi\in\mathbb{R}^{d}.
\]

%
\begin{theorem}\label{boundaryRegularity}
Let $u$ be a weak solution of
\begin{eqnarray*}
Mu &=& f(x) \in\Gamma,
\\
u - g &\in& H_{0}^{1,2}(\Gamma).
\end{eqnarray*}
%
Suppose that the ellipticity condition holds. Let $f \in W^{k,2}(\Gamma
)$, $g \in W^{k+2,2}(\Gamma)$. Let $\Gamma$ be of class $C^{k+2}$, and
let the coefficients of $M$ be of class $C^{k+1}(\bar{\Gamma})$. Then
\[
\Vert u \Vert_{W^{k+2,2}(\Gamma)} \leq c \bigl( \Vert f\Vert_{W^{k,2}(\Gamma)}+\Vert g
\Vert_{W^{k+2,2}(\Gamma
)} \bigr)
\]
with $c$ depending on $\lambda, d$, $\Gamma$ and on the $C^{k+1}$-norms
for the $a^{i,j}$.
\end{theorem}

\begin{pf}
The proof of Lemma~8.3.3 is given on page 207 in \cite{jost1998partielle}.
\end{pf}

\section{Expected signature of planar Brownian motion up to the first
exit time from a bounded domain}\label{section2}
Recall that $E=\mathbb{R}^{d}$, and so it has a canonical basis
$(e_{1}, \ldots, e_{d})$. Let $ ( B_{t} ) _{t\geq0}$ denote a
standard $d$-dimensional Brownian motion on $E$ under a probability
space $(\Omega, P^{z}, \mathcal{F})$ with its canonical filtration
$\mathcal{F} =  (\mathcal{F}_{t} )_{t\geq0}$ where
$P^{x}(B_{0} = z) = 1$ and $z \in E$.

\begin{definition}
Let $\Gamma$ be a domain (a connected open subset) in $E$. Then $\tau
_{\Gamma}=\inf\{t\geq0\dvtx B_{t}\in\Gamma^{c}\}$ is the first exit
time of
Brownian motion from $\Gamma$.
\end{definition}

The (Stratonovich) signature is defined for almost every Brownian path
$B$, and for all pairs of times $(s, t)$. We are interested in the
random signature $S(B_{[0, \tau_{\Gamma}]})$ of the Brownian path up to
the first exit time from the domain $\Gamma$.

\begin{definition}[(Expected signature of Brownian motion)]  Assume
$\tau_{\Gamma} < \infty$ a.s. and the componentwise integrability
of $S ( B_{ [ 0,\tau_{\Gamma} ] } )$. We denote by
$\Phi_{\Gamma}(z)$ the expected signature of Brownian motion starting
at $z$ and stopped upon the first exit time $\tau_{\Gamma}$ from a
domain $\Gamma$, that is,
%
\begin{equation}
\Phi_{\Gamma}(z)=\mathbb{E}^{z} \bigl[S(B_{  [ 0,\tau_{\Gamma} ] })
\bigr].
\end{equation}
\end{definition}


\begin{remark}
For the case of Brownian motion in the upper half plane $\mathbb{H}$,
the projection $\pi^{(1)}$ of the signature at the exit time onto the
horizontal axis has a Cauchy distribution, and therefore does not have
finite expectation. Similarly the projection $\pi^{(1, 1)}$ is positive
and has a $\frac{1}{2}$-stable distribution and so more obviously does
not have finite expectation.
\end{remark}

%

In order to discuss the expected signature of Brownian motion further, we
introduce an auxiliary function $\Psi$ mapping $E$ to $T((E))$.

\begin{definition}
We denote by $\Psi(z)$ the expected signature of a Brownian motion started
at $0$, run until it leaves the ball $\mathbb{D}(0, |z|)$, and
conditioned to exit the ball at the point $z$. In formula,
\[
\Psi(z)=\mathbb{E}^{0} \bigl[S(B_{ [ 0,\tau_{\mathbb{D}(0, |z|)} ]
})|B_{\tau_{\mathbb{D}(0, |z|)}}=z
\bigr].
\]
\end{definition}

In particular, $\Psi(0) = \mathbf{1} = (1, 0, 0, \ldots)$. The function
$\Psi$ plays an important role in the rest of the paper. In part, its
importance comes from the strong Markov property. In Lemma~\ref
{Lemma_Psi_smooth}, we easily prove that $\Psi$ is well defined, and a
smooth function of $z$. A more serious challenge is to show that the
boundedness of the domain $\Gamma$ can guarantee the existence and
twice differentiability of $\Phi_{\Gamma}$. Then we will be in the
position to derive a PDE which characterizes $\Phi_{\Gamma}$.

In the case of the disk, which is clearly a bounded domain, the PDE is
so explicit that one
can use it to identify the solution as a combination of polynomials in an
explicit manner. The solution to this PDE contains an enormous amount
of information about Brownian motion in the disk, and, for example, it
easily gives the moments of the L\'evy area for Brownian motion stopped
at the first exit time from the disk in a quite explicit form; see \cite{NiHaoThesis}.

\subsection{The componentwise smoothness of \texorpdfstring{$\Phi_{\Gamma}$}{$Phi_{Gamma}$}}

We start by proving that the signature of Brownian motion upon its
first exit time $\tau_{\Gamma}$ from a bounded domain $\Gamma$ has
finite expectation, and so $\Phi_{\Gamma}$ is well defined. We will
then go on to discuss the componentwise smoothness of $\Phi_{\Gamma}$.
The expectation can be controlled by using extension of Lepingle's BDG
inequality obtained in \cite{BDG}.

\begin{definition}
We say that $F\dvtx \mathbb{R}^{+} \rightarrow\mathbb{R}^{+}$ is moderate if:
\begin{longlist}[(1)]
\item[(1)] $F$ is continuous and increasing;
\item[(2)] $F(x)=0$ if and only if $x=0$;
\item[(3)] for some (and thus for every $\alpha>1$)
\[
\sup_{x > 0}\frac{F(\alpha x)}{F(x)}< \infty.
\]
\end{longlist}
\end{definition}

%
\begin{theorem}
\label{Thm_estimate_es} Let $M$ be a continuous, $\mathbb{R}^{d}$-valued
local martingale starting at 0 and $S^{n}(M)$ be the truncated
Stratonovich signature of $M$ up to level $n$ viewed as a path in
$T^{(n)}(\mathbb{R}^{d})$. Then for any moderate function $F$, there
exists $%
C_{i}=C_{i}(n,F,d,|\cdot|)$ for $i = 1, 2$ such that for all stopping times
$\tau$,
\begin{eqnarray*}
&&\mathbb{E} \biggl(F \biggl( \biggl|\int_{0<u_{1}<u_{2}<\cdots<u_{n}<\tau} \circ
dM_{u_{1}}\circ\cdots \circ dM_{u_{n}} \biggr|^{1/n}
\biggr) \biggr)
\\
&&\qquad\leq C_{1}\mathbb{E} \bigl(F \bigl( \bigl\Vert S^{n}(M)
\bigr\Vert _{p\mbox{-}\mathrm{var};[0,\tau]} \bigr) \bigr)
\\
&&\qquad\leq C_{2}\mathbb{E} \bigl( F \bigl( \bigl|\langle M
\rangle_{\tau} \bigr|^{
{1}/{2}%
} \bigr) \bigr).
\end{eqnarray*}
\end{theorem}

The proof can be found in \cite{BDG}.

\begin{corollary}
\label{boundness} Let $\Gamma$ be a bounded domain. Then for every $n
\in
\mathbb{N}^{+}$, there exists a constant $C=C(n)$ such that for every
$z \in\Gamma$,
\begin{eqnarray*}
&&\mathbb{E}^{z} \biggl[ \biggl\llvert {%
\int\cdots
\int}_{0<u_{1}<u_{2}<\cdots<u_{n}<\tau
_{\Gamma}}\circ dB_{u_{1}}\circ\cdots\circ
dB_{u_{n}} \biggr\rrvert \biggr]
\\
&& \qquad\leq C(n)\mathbb{E}^{z} \bigl[\bigl\Vert S^{n}(B)\bigr\Vert
_{p\mbox{-}\mathrm{var};[0,\tau]}^{n} \bigr]
\\
&& \qquad\leq C(n)\sup_{z \in\Gamma} \bigl\{\mathbb{E}^{z}
\bigl[ \tau_{\Gamma}^{n/2} \bigr] \bigr\} < +\infty.
\end{eqnarray*}
\end{corollary}

\begin{pf}
Applying Theorem~\ref{Thm_estimate_es} to the case with $M_{t}=B_{t}$
and $%
F(x)=x^{n}$, we get
\begin{eqnarray*}
\mathbb{E}^{z} \biggl[ \biggl\llvert {%
\int\cdots
\int}_{0<t_{1}<t_{2}<\cdots<t_{n}<\tau_{\Gamma
}}\circ dB_{t_{1}}\circ\cdots\circ
dB_{t_{n}} \biggr\rrvert \biggr] &\leq& C(n)\mathbb{E}^{z}
\bigl[ \bigl\Vert S^{n}(B) \bigr\Vert _{p\mbox{-}\mathrm{var};[0,\tau]}^{n} \bigr]
\\
&\leq& C(n)\mathbb{E}^{z} \bigl[\tau_{\Gamma}^{n/2}
\bigr].
\end{eqnarray*}
The only thing we need to show that is ${\sup}_{z \in\Gamma}\{
\mathbb{E}^{z}[\tau_{\Gamma}^{n/2}]\}
< +\infty$. It can be shown that there exists a positive number $\alpha
> 0$
such that
\[
\mathbb{E}^{0} \bigl[e^{\alpha\tau_{\mathbb{D}}} \bigr] < + \infty,
\]
and due to the Brownian scaling property we have
\[
\forall r>0\qquad \mathbb{E}^{0} \bigl[e^{\alpha\tau_{\mathbb{D}}} \bigr]=
\mathbb {E}^{0} \bigl[e^{%
{\alpha}/{r^{2}}\tau_{\mathbb{D}(0,r)}} \bigr].
\]
Since $\Gamma$ is bounded, there exists a positive constant $R >0$ such that
$\Gamma\subseteq\mathbb{D}(0, R)$. Thus $0 \leq\tau_{\Gamma} \leq
\tau_{
\mathbb{D}(0, R)}$. This implies that
\[
\mathbb{E}^{z} \bigl[e^{{\alpha}/{(4R^{2})}\tau_{\Gamma}} \bigr]\leq\mathbb
{E}^{z} \bigl[e^{%
{\alpha}/{(4R^{2})}\tau_{\mathbb{D}(z, 2R)}} \bigr] = \mathbb {E}^{0}
\bigl[e^{{%
\alpha}/{(4R^{2})}\tau_{\mathbb{D}(0, 2R)}} \bigr]= \mathbb{E}^{0} \bigl[e^{\alpha\tau
_{%
\mathbb{D}}}
\bigr] < \infty.
\]
Thus $\tau_{\Gamma}$ has finite moments, since for every $z \in\Gamma$,
\[
\frac{1}{n!} \biggl(\frac{\alpha}{4R^{2}} \biggr)^{n}
\mathbb{E}^{z} \bigl[\tau _{\Gamma}^{n} \bigr]\leq
\mathbb{E}^{z} \bigl[e^{{\alpha}/{(4R^{2})}\tau_{\Gamma
}} \bigr]\leq\mathbb{E}^{0}
\bigl[e^{\alpha\tau_{
\mathbb{D}}} \bigr] < \infty.
\]
Then it follows that
\[
\forall n \in\mathbb{N}\qquad\sup_{z \in\Gamma} \bigl\{
\mathbb{E}^{z} \bigl[\tau _{\Gamma}^{n} \bigr] \bigr\}
< \infty
\]
and
\[
\sup_{z \in\Gamma} \bigl\{\mathbb{E}^{z} \bigl[
\tau_{\Gamma}^{n/2} \bigr] \bigr\} \leq\sqrt{ \sup
_{z \in\Gamma} \bigl\{\mathbb{E}%
^{z} \bigl[
\tau_{\Gamma}^{n} \bigr]} \bigr\} < \infty.
\]
Now our proof is complete.
\end{pf}

In the rest of this subsection, we are going to discuss the smoothness
of $%
\Phi_{\Gamma}$ in componentwise sense providing that $\Psi$ is a smooth
function, which is proved later in Lemma~\ref{Lemma_Psi_smooth}.

\begin{theorem}\label{Phi_theorem}
Suppose that $\Gamma$ is a nonempty bounded domain in $E$. Then the following
statements are true:
\begin{longlist}[(1)]
\item[(1)] $\Phi_{\Gamma}$ is a well defined function and moreover $\Phi
_{\Gamma}\in L^{1}$;
\item[(2)] $\Phi_{\Gamma}$ is twice differentiable in componentwise
sense, that is,
for all index $I$, $\pi^{I}\circ\Phi_{\Gamma}$ is twice differentiable.
\end{longlist}
\end{theorem}

\begin{pf}
We start with proving the first statement. By Corollary~\ref
{boundness}, it is easy to see that $\Phi_{\Gamma}$ is well defined
and uniformly bounded in $\Gamma$, since for every index~$I$,
there exists a constant $C>0$, such that
\[
\sup_{z \in\Gamma} \bigl\vert\pi^{I} \bigl(
\Phi_{\Gamma}(z) \bigr) \bigr\vert\leq \sup_{z
\in\Gamma}
\mathbb{E}^{z} \bigl[ \bigl\llvert \pi^{I}
\bigl(S(B_{[0, \tau_{\Gamma}]}) \bigr) \bigr\rrvert \bigr]< C.
\]
Obviously $\Phi_{\Gamma}$ is a measurable function. Furthermore, $\pi
^{I}(\Phi_{\Gamma})$ has a compact support. Hence for every index $I$,
\[
\int_{\Gamma} \bigl|\pi^{I} \bigl(\Phi_{\Gamma}(z)
\bigr) \bigr|\,dz \leq C A(\Gamma)<\infty,
\]
where $A(\Gamma)$ is the volume of $\Gamma$. So $\Phi_{\Gamma} \in
L^{1}$. Thus the proof of the first statement is complete. Now the only
thing left to prove is that $\Phi_{\Gamma}$ is twice differentiable.

For every $\varepsilon> 0$, let $\Gamma_{\varepsilon}=\{z \in
\Gamma| \operatorname{dist}(z,\partial\Gamma)>\varepsilon\}$. The Markov property of
the expected signature of
Brownian motion, which is described in detail in Lemma~\ref
{Mean_value_property_lemma}, ensures that for every $z \in\Gamma
_{\varepsilon} $,
\[
\Phi_{\Gamma}(z) = \frac{1}{d\omega_{d}r^{d-1}} \int_{\partial\mathbb
{D}(0, r)}
\Psi(y)\otimes\Phi_{\Gamma}(z+y)\,d\sigma(y),
\]
where $r<\frac{\varepsilon}{2}$, $\omega_{d}$ is the volume of the unit
ball, and $\sigma(y)$ is the $d-1$-dimensional surface measure.

Then this implies that 
\begin{equation}
\label{Eq_6} \Phi_{\Gamma}(z) = \int_{0}^{\infty}
\biggl(\frac{1}{d\omega_{d}r^{d-1}} \int_{\partial\mathbb{D}(0, r)}\Psi(y)\otimes
\Phi_{\Gamma}(z+y)\,d\sigma(y) \biggr)K_{\varepsilon}(r)\,dr,
\end{equation}
for a smooth distribution $K_{\varepsilon}(r)$ with compact support
$[0,\frac{%
\varepsilon}{2}]$. Let $F_{\varepsilon}$ be a map from $E$ to $T((E))$
defined by
\[
F_\varepsilon(z)=\Psi(-z)K_{\varepsilon} \bigl(|z| \bigr).
\]
Rewriting (\ref{Eq_6}) we have
\[
\Phi_{\Gamma}(z)=\int_{\Gamma}F_{\varepsilon}(z-y)
\otimes\Phi_{\Gamma}(y)\,dy = F_{\varepsilon}*\Phi_{\Gamma}(z),
\]
where $*$ is the convolution. Since $\Psi$ is smooth and $K_{\varepsilon
}$ is a smooth
function with compact support, $F_{\varepsilon}$ is a smooth function
with compact support. It is easy to show that
\[
\Vert F_{\varepsilon}\Vert _{L^{1}}+\Vert\! \bigtriangledown
F_{\varepsilon}\!\Vert _{L^{1}}+\Vert \bigtriangleup F_{\varepsilon
}
\Vert _{L^{1}}<+\infty.
\]
On the other hand, $\Phi_{\Gamma}\in L^{1}$ as well, and so we have
\[
\Vert F_{\varepsilon}*\Phi_{\Gamma}\Vert _{L^{1}}+\Vert\!
\bigtriangledown F_{\varepsilon}*\Phi_{\Gamma}\!\Vert _{L^{1}}+
\Vert \bigtriangleup F_{\varepsilon}*\Phi_{\Gamma}\Vert
_{L^{1}}<+\infty.
\]
Thus $F_{\varepsilon}*\Phi_{\Gamma}$ is twice differentiable, since $%
F_{\varepsilon}*\Phi_{\Gamma}, (\bigtriangledown
F_{\varepsilon})*\Phi_{\Gamma}\mbox{ and } (\bigtriangleup
F_{\varepsilon})*\Phi_{\Gamma}\in L^{1}$. Moreover, because $\Gamma=
{\bigcup}_{\varepsilon> 0}\Gamma_{\varepsilon}$, $\Phi_{\Gamma
}$ is
smooth on $\Gamma$.
\end{pf}

\begin{remark}
Actually following this proof we can show that $\Phi_{\Gamma}$ is
infinitely differentiable in a bounded domain. Alternatively, since we
have proved that $\Phi_{\Gamma}$ is twice differentiable so that our
PDE provides the integral representation for $\Phi_{\Gamma}$, this also
implies that $\Phi_{\Gamma}$ is infinitely differentiable.
\end{remark}

\subsection{\texorpdfstring{\spaceskip=0.7em plus 0.05em minus 0.02em Properties of the
expected signature of stopped Brownian motion}{{Properties of the
expected signature of stopped Brownian motion}}}
\mbox{Throughout} this subsection, suppose that $\Gamma$ is a bounded domain.
In view of Theorem~\ref{Phi_theorem} $\Phi_{\Gamma}(z)$ is well defined
and twice differentiable. Moreover $\Phi_{\Gamma}(z)$ has the following
properties.

\subsubsection{Translation invariance}

\begin{lemma}[(Translation invariance)]
Let $z + \Gamma= \{z+y \vert y \in\Gamma\}$. Then
\begin{eqnarray*}
\mathbb{E}^{z} \bigl[S(B_{[0,\tau_{z+\Gamma}]}) \bigr] &=&
\mathbb{E}^{0} \bigl[S ( B_{%
 [ 0,\tau_{\Gamma} ] } ) \bigr],
\\
\mathbb{E}^{z} \bigl[S(B_{[0,\tau_{z+\Gamma}]})|B_{\tau_{z+\Gamma}}=z+a \bigr]
&=&%
\mathbb{E}^{0} \bigl[S(B_{[0,\tau_{\Gamma}]})|B_{\tau_{\Gamma}}=a
\bigr],
\end{eqnarray*}
where $a\in\partial\Gamma$.
\end{lemma}

\subsubsection{Scaling property}

\begin{lemma}[(Scaling property)]
For $n\in\mathbb{N}\mbox{ and }\varepsilon\in[0,+\infty)$, we have that
\begin{eqnarray*}
\Phi_{\varepsilon\mathbb{D}}(0) &=&\delta_{\varepsilon} \bigl(\Phi_{\mathbb
{D}%
}(0)
\bigr),\qquad \mbox{that is, }\rho_{n} \bigl(\Phi_{\varepsilon\mathbb{D}}(0)
\bigr) = \varepsilon ^{n}\rho _{n} \bigl(\Phi_{\mathbb{D}}(0)
\bigr);
\\
\Psi(\varepsilon) &=& \delta_{\varepsilon} \bigl(\Psi(1) \bigr), \qquad
\mbox{that is, }\rho_{n} \bigl(\Psi(\varepsilon) \bigr) =
\varepsilon^{n}\rho_{n} \bigl(\Psi(1) \bigr).
\end{eqnarray*}
\end{lemma}

\subsubsection{Rotation property}
Let $O(d)$ denote the rotation group in $d$ dimensions. Let us
introduce the rotation operator $\bar{\delta}_{\theta}$ on $T((E))$,
where $\theta\in O(d)$. The rotation operator $\bar{\delta}_{\theta}$
is defined as follows: for every $a = (a_{0}, a_{1}, \ldots)\in T((E))$,
\[
\bar{\delta}_{\theta}(a) = \bigl(a_{0}, \theta
a_{1}, \ldots, \theta^{\otimes
n} a_{n}, \ldots \bigr),
\]
where $\cdot^{\otimes\cdot}$ is the Kronecker power, and $bc$ is the
matrix multiplication of $b$ and $c$.

\begin{lemma}[(Rotation property)]
\label{rotation_lemma}
For every fixed $x \in\partial\mathbb{D}$ and some neighborhood $U$
contained in $\partial\mathbb{D}$ around $x$, there is a smooth map
$u$ from $U$ into $O(d)$ so that for every $y \in U$, $u(y)e_{1} = y$. Then
\[
\Psi(y) = \bar{\delta}_{u(y)} \bigl(\Psi(e_{1}) \bigr).
\]
%
\end{lemma}

\begin{pf}
Let $B_{t}$ be the Brownian motion starting at the origin and stopped for
the first exit time of the unit disk on condition that the exit point
is $%
e_{1}$. Then for every $y \in U$, $\tilde{B}_{t} = u(y)B_{t}$ is the
Brownian motion
starting at the origin and stopped at the first exit time from the unit disk
on condition that the exit point is~$y$. It is obvious that $\tilde{B}_{t}$
is just a linear transformation of $B_{t}$. 
Then immediately it follows that for every $y \in U$,
\[
\rho^{n} \bigl(\Psi(y) \bigr) = u(y)^{\otimes n}\rho^{n}
\bigl(\Psi(e_{1}) \bigr)
\]
or equivalently
\[
\Psi(y) = \bar{\delta}_{u(y)} \bigl(\Psi(e_{1}) \bigr).
\]
\upqed\end{pf}

\subsubsection{An application of the strong Markov property}

\begin{lemma}[(Strong Markov property)]
For every $\varepsilon>0$ such that\break \mbox{$\mathbb{D}(z, \varepsilon
)\subseteq
\Gamma$},
%
\begin{eqnarray}
\label{Eq_Markovian_property_1}\qquad \Phi_{\Gamma}(z)&=& \frac{1 }{d \omega_{d} \varepsilon^{d-1}}
\nonumber
\\[-8pt]
\\[-8pt]
\nonumber
&&{}\times\int_{\partial%
\mathbb{D}(0, \varepsilon)} \mathbb{E}^{z}
\bigl[S(B_{[0,\tau_{
\mathbb{D}(z, \varepsilon)}]})|B_{\tau_{ \mathbb{D}(z, \varepsilon
)}}=z+y \bigr] \otimes\Phi
_{\Gamma}(z+y)\,d\sigma(y),
\end{eqnarray}
where $\partial\mathbb{D}(0, \varepsilon)$ denotes the boundary of
$\mathbb{D}(0, \varepsilon)$, $\omega_{d}$ is the volume of
$d$-dimensional unit ball and $\sigma$ is the $d-1$-dimensional surface measure.
\end{lemma}

\begin{pf}
By Chen's identity (Theorem~\ref{Chen}) and the fact that $S(B_{[0,\tau_{
\mathbb{D}(z,\varepsilon)}]})$ is $\mathcal{F}_{\tau_{\mathbb{D}%
(z,\varepsilon)}}$-measurable, we obtain
\[
\mathbb{E}^{z} \bigl[S(B_{[0,\tau_{\Gamma}]})|\mathcal{F}_{\tau_{\mathbb
{D}(z,\varepsilon)}}
\bigr] = S(B_{[0,\tau_{z+\varepsilon\mathbb{D}}]})\otimes \mathbb{E} \bigl[S(B_{[\tau_{\mathbb{D}(z,\varepsilon)},\tau_{\Gamma
}]})|
\mathcal{F}_{\tau_{\mathbb{D}(z,\varepsilon)}} \bigr],
\]
where $\mathbb{D}(z,\varepsilon) \subseteq\Gamma$ and $\mathcal{F}$ is
defined as before, that is, the filtration generated by the Brownian motion.

As it is known that $B$ has the strong Markov property,
\[
\mathbb{E}^{z} \bigl[S(B_{[\tau_{\mathbb{D}(z,\varepsilon)},\tau_{\Gamma
}]})|\mathcal{F}%
_{\tau_{\mathbb{D}(z,\varepsilon)}}
\bigr] = \mathbb{E}^{z} \bigl[S(B_{[\tau
_{\mathbb{D}%
(z,\varepsilon)},\tau_{\Gamma}]})|B_{\tau_{\mathbb{D}(z,\varepsilon)}}
\bigr] = \Phi_{\Gamma}(B_{\tau_{\mathbb{D}(z,\varepsilon)}}),
\]
which implies that
%
\begin{equation}
\label{Eq_1} \mathbb{E}^{z} \bigl[S(B_{[0,\tau_{\Gamma}]})|
\mathcal{F}_{\tau_{\mathbb{D}%
(z,\varepsilon)}} \bigr] = S(B_{[0,\tau_{\mathbb{D}(z,\varepsilon)}]})\otimes
\Phi_{\Gamma}(B_{\tau_{\mathbb{D}(z,\varepsilon)}}).
\end{equation}
By the tower property, we have
%
\begin{equation}
\Phi_{\Gamma}(z) = \mathbb{E}^{z} \bigl[\mathbb{E}^{z}
\bigl[S(B_{[0,\tau_{\Gamma
}]} )|%
\mathcal{F}_{\tau_{\mathbb{D}(z,\varepsilon)}} \bigr] \bigr].
\end{equation}
Substituting equation (\ref{Eq_1}) into it, we obtain
\begin{eqnarray*}
\Phi_{\Gamma}(z) &=& \mathbb{E}^{z} \bigl[S(B_{[0, \tau_{z+\varepsilon\mathbb
{D}%
}]})
\otimes\Phi_{\Gamma}(B_{\tau_{\mathbb{D}(z,\varepsilon)}}) \bigr]
\\
&=&\frac{1 }{d \omega_{d} \varepsilon^{d-1}}
\\
&&{}\times\int_{\partial\mathbb{D}(0,
\varepsilon)}\mathbb{E}^{z}
\bigl[S(B_{[0,\tau_{z+\varepsilon\mathbb{D}%
}]})|B_{\tau_{z+\varepsilon\mathbb{D}}}=z+y \bigr]\otimes
\Phi_{\Gamma
}(z+y)\,d \sigma(y).
\end{eqnarray*}
\upqed\end{pf}

\subsubsection{The smoothness of \texorpdfstring{${\Psi}$}{${Psi}$}}

In this subsection, we focus on the discussion of the componentwise
smoothness of $\Psi$. We start with the proof of the finiteness of $\Psi
(e_{1})$, and then show that the function $\Psi$ is smooth. We end with
the derivation of the first two gradings of $\Psi$, which is important
for us to derive the PDE for $\Phi$ later.

\begin{lemma}
$\Psi(e_{1})$ is well defined.
\end{lemma}

\begin{pf}
Let $\tilde{B}_{t}$ be a $d$-dimensional Brownian motion starting at
zero in the unit ball. Let $M_{t}=\tilde{B}_{\tau_{\mathbb{D}}\wedge
t}$ be the stopped process. It is obvious from BDG inequality that the
exit time $\tau_{\mathbb{D}}$ of $\tilde{B}_{t}$ from the ball (or any
bounded set) has finite moments of all orders.

By the theorem of the equivalence of homogeneous norms on $G^{N}(E)$
(Lem\-ma~\ref{EquivalenceOf Homogeneous norms}), we have
\[
\biggl\llvert \int_{0<u_{1}<u_{2}<\cdots<u_{N}<\tau_{\mathbb{D}} }\circ dM_{u_{1}}\circ
\cdots \circ dM_{u_{N}} \biggr\rrvert \leq C \bigl\llVert
S^{N}(M) \bigr\rrVert _{p\mbox{-}\mathrm{var};[0,\tau_{\mathbb{D}}]}^{N}.
\]
Thus it holds that
\begin{eqnarray*}
&&\mathbb{E}^{0} \biggl[ \biggl\llvert \int_{0<u_{1}<u_{2}<\cdots<u_{N}<\tau
}
\circ dM_{u_{1}}\circ\cdots\circ dM_{u_{N}} \biggr\rrvert \Big|
M_{\tau_{\mathbb
{D}}}=e_1%
 \biggr]
\\
&&\qquad\leq C \mathbb{E}^{0} \bigl[ \bigl\llVert S^{N }(M)
\bigr\rrVert _{p\mbox{-}\mathrm{var};[0,\tau_{\mathbb{D}}]}^{N}| M_{\tau_{\mathbb{D}}}=e_{1}
\bigr].
\end{eqnarray*}
By the definition of the Carnot--Caratheodory norm, it is obvious that
for every $\theta\in O(d)$
\[
\bigl\Vert S^{N}(\theta M) \bigr\Vert _{p\mbox{-}\mathrm{var};[0,\tau]}
\]
does not depend on $\theta$. Therefore
\[
\mathbb{E}^{0} \bigl[ \bigl\llVert S^{N }(M) \bigr\rrVert
_{p\mbox{-}\mathrm{var};[0,\tau_{%
\mathbb{D}}]}^{N}| M_{\tau_{\mathbb{D}}}=e_{1} \bigr] =
\mathbb {E}^{0} \bigl[%
 \bigl\llVert S^{N }(M)
\bigr\rrVert _{p\mbox{-}\mathrm{var};[0,\tau_{\mathbb
{D}}]}^{N} \bigr],
\]
because there is nothing significant in choosing the exit point to be $e_{1}$
providing the norm used is invariant under rotation; using extension of
Lepingle's BDG inequality obtained in \cite{BDG} one can control the $p$-variation of a martingale in terms of its bracket process, and so there
exists a positive constant\vadjust{\goodbreak} $C > 0$ such that
\[
\mathbb{E}^{0} \bigl[ \bigl\Vert S^{N}(M) \bigr\Vert
_{p\mbox{-}\mathrm{var};[0,\tau_{\mathbb
{D}}]}^{N} \bigr] \leq C \mathbb{E}^{0} \bigl[
\tau_{\mathbb{D}}^{{N}/{2}} \bigr].
\]
Finally we have that there exists a constant $C \geq0$ such that
\begin{eqnarray*}
\bigl|\rho_{N} \bigl(\Psi(1) \bigr) \bigr|&=& \biggl\llvert
\mathbb{E}^{0} \biggl[ \int_{0<u_{1}<u_{2}<\cdots<u_{N}<\tau_{\mathbb{D}}}\circ
dM_{u_{1}}\circ \cdots\circ dM_{u_{N}} \Big\vert
M_{\tau_{\mathbb{D}}}=1 \biggr] \biggr\rrvert
\\
&\leq& \mathbb{E}^{0} \biggl[ \biggl\llvert \int_{0<u_{1}<u_{2}<\cdots<u_{N}<\tau_{%
\mathbb{D}}}
\circ dM_{u_{1}}\circ\cdots\circ dM_{u_{N}} \biggr\rrvert \Big|
M_{\tau_{%
\mathbb{D}}}=1 \biggr]
\\
&\leq& C \mathbb{E}^{0} \bigl[\tau_{\mathbb{D}}^{{N}/{2}}
\bigr]<+\infty.
\end{eqnarray*}
\upqed\end{pf}

\begin{lemma}
\label{Lemma_Psi_smooth} $\Psi$ is a componentwise smooth function in
$E\setminus\{0\}$,
that is, for every index $I$, $\pi^{I}(\Psi)$ is a smooth function.
\end{lemma}

\begin{pf}
We recall the definition of $\Psi(x)$. It is clear that if $\alpha>0$
and $\theta$ is a rotation, then
\begin{eqnarray*}
\Psi \bigl(\theta(\alpha x) \bigr) &=& \delta_{\alpha} \bigl(\Psi(\theta
x) \bigr)
\\
&=& \delta_{\alpha} \bigl(\bar{\delta}_{\theta} \bigl(\Psi(x)
\bigr) \bigr).
\end{eqnarray*}
In particular the map $(\alpha, \theta)\mapsto\Psi(\theta(\alpha x))$
is smooth and defined on $\mathbb{R}^{+} \times O(d)$ to $T((E))$. Fix
$x$ on the sphere and some small neighborhood $U$ contained in
$\partial\mathbb{D}$ around $x$ chosen so that there is a smooth map
$u$ from $U$ into $O(d)$ so that for every $y \in U$, $u(y)(e_{1}) =
y$. We observe that on $\mathbb{R}^{+} \times U$,
\[
\Psi(\alpha y) = \delta_{\alpha} \bigl(\bar{\delta}_{u(y)} \bigl(
\Psi(e_{1}) \bigr) \bigr).
\]
As the right-hand side is the composition of two smooth maps $\delta_{\alpha}$ and
$\bar{\delta}_{u(x)}$, the left-hand side must be smooth on $\mathbb{R}^{+} \times
U$. Since $x$ is an arbitrary point in the sphere, we have proved the result.
\end{pf}


\begin{lemma}
\label{Psi_remark} $\Psi$ is a twice continuously differentiable
function in $E$ in componentwise sense.
Moreover,
\begin{eqnarray*}
\pi_{2} \bigl(\Psi(z) \bigr)&=&1+\sum_{i=1}^{d}z_{i}e_{i}+
\sum_{i=1}^{d}\frac{1}{2}
z_{i}^{2}e_{i}\otimes
e_{i},
\\
\frac{\partial\Psi(0)}{\partial z_{i}} &=&e_{i} \qquad\mbox{for }i=1, 2, \ldots, d\quad
\mbox{and}
\\
\Delta\Psi(0) &=&\sum_{i=1}^{d}e_{i}
\otimes e_{i}.
\end{eqnarray*}
\end{lemma}

\begin{pf}
The term of tensor degree one in $\Psi$ is the expected increment of the
conditioned path, so it is not random and $\rho_{1}(\Psi(z))=z$. The
expectation of the L\'evy area of the Brownian motion conditioned on
leaving the
disk at\vadjust{\goodbreak} any particular point equals to zero due to the symmetry of the
Brownian motion. Thus the term $\rho_{2}(\Psi(z))$ only contains the
symmetric ``increment squared'' part, thus we have
\[
\pi_{2} \bigl(\Psi(z) \bigr)=1+\sum_{i=1}^{d}z_{i}e_{i}+
\sum_{i=1}^{d}\frac{1}{2}%
z_{i}^{2}e_{i}
\otimes e_{i}.
\]
This implies that $\pi_{2}(\Psi(z))$ is smooth in $E$, and moreover it
follows that
\begin{eqnarray*}
\frac{\partial\pi_{2}(\Psi(0))}{\partial z_{i}} &=&e_{i}\qquad\mbox{ for }i=1, 2, \ldots, d\quad
\mbox{and}
\\
\Delta\pi_{2} \bigl(\Psi(0) \bigr) &=&\sum
_{i=1}^{d}e_{i} \otimes e_{i}.
\end{eqnarray*}
Now let us focus on $\rho_{n} \circ\Psi$ for any $n \geq3$. By Lemma~\ref{Lemma_Psi_smooth}, $\Psi$ is smooth in $E\setminus\{0 \}$; the
scaling property of $\Psi$ can guarantee the smoothness of $\Psi$ at
the origin, which we explain in more detail as follows. By the
definition of $\Psi$, $\Psi(0) = \mathbf{1}$. By the scaling property
of $\rho_{n}\circ\Psi$, it is easy to verify that for every integer $n
\geq3$,
\begin{eqnarray*}
\frac{\partial\rho_{n}(\Psi(0))}{\partial z_{i}} &=& \lim_{\alpha
\rightarrow0} \frac{\rho_{n}(\Psi(\alpha e_{i}))-\rho_{n}(\Psi
(0))}{\alpha} = \lim
_{\alpha\rightarrow0} \alpha^{n-1}\rho_{n} \bigl(\Psi
(e_{i}) \bigr) = 0;
\\
\frac{\partial\rho_{n}(\Psi(z_{i}e_{i}))}{\partial z_{i}} &=& \frac
{\partial z_{i}^{n}\rho_{n}(\Psi(e_{i}))}{\partial z_{i}} = n z_{i}^{n-1}
\rho_{n} \bigl(\Psi(e_{i}) \bigr)\qquad\forall z_{i}
\in(-1, 1);
\\
\frac{\partial^{2} \rho_{n}(\Psi(0))}{\partial z_{i}^{2}} &=&
\lim_{\alpha\rightarrow0}\frac{1}{\alpha} \biggl(
\frac{\partial\rho
_{n}(\Psi(\alpha e_{i}))}{\partial z_{i}}-
 \frac{\partial\rho_{n}(\Psi
(0))}{\partial z_{i}} \biggr)
\\
&=& \lim_{\alpha\rightarrow0} n \alpha^{n-2}\rho_{n}
\bigl(\Psi(e_{i}) \bigr) = 0.
\end{eqnarray*}
%
Thus it follows that for every integer $n \geq3$,
\[
\Delta\rho_{n} \bigl(\Psi(0) \bigr) = 0.
\]
The proof is complete.
\end{pf}

\subsection{The PDE for the expected signature of multi-dimensional
Brownian motion up to the first exit time from a bounded domain}

In this section, our goal is to derive the PDE for the expected
signature of
Brownian motion up to the first exit time from a bounded domain.

\begin{lemma}
\label{Mean_value_property_lemma} For every $z \in\Gamma$ and every $%
\varepsilon>0$ sufficiently small such that $\mathbb{D}(z, \varepsilon)
\subset\Gamma$, it holds that 
\begin{equation}
\Phi_{\Gamma}(z)= \frac{1}{d \omega_{d}\varepsilon^{d-1}}\int_{\partial{
\mathbb{D}(0, \varepsilon)}}\Psi(y)
\otimes\Phi_{\Gamma} (z+y)\,d\sigma(y),
\end{equation}
where $\omega_{d}$ is the volume of the $d$-dimensional unit ball, and $
\sigma$ is the $(d-1)$-dimensional surface measure.
\end{lemma}

\begin{pf}
From the strong Markov property, we have
%
\begin{eqnarray}
\label{Eq_Markovian_property} \Phi_{\Gamma}(z) &=&\frac{1}{d \omega_{d}\varepsilon^{d-1}}
\nonumber
\\[-8pt]
\\[-8pt]
\nonumber
&&{}\times\int_{\partial{%
\mathbb{D}(0, \varepsilon)}}\mathbb{E}^{z}
\bigl[S(B_{[0,\tau_{\mathbb{D}%
(z,\varepsilon)}]} )|B_{\tau_{\mathbb{D}(z,\varepsilon)}}=z+y \bigr]\otimes
\Phi_{\Gamma}(z+y)\,d\sigma(y).
\end{eqnarray}
By the translation invariance property, substituting
\[
\Psi(y) = \mathbb{E}^{z} \bigl[S(B_{[0,\tau_{\mathbb{D}(z,\varepsilon)}]}
)|B_{\tau_{\mathbb{D}(z,\varepsilon)}}=z+y \bigr]
\]
into (\ref{Eq_Markovian_property}), we obtain that
\[
\Phi_{\Gamma}(z) = \frac{1}{d \omega_{d}\varepsilon^{d-1}}\int_{\partial
{%
\mathbb{D}(0, \varepsilon)}}
\Psi(y)\otimes\Phi_{\Gamma} (z+y)\,d\sigma(y),
\]
where $\varepsilon$ is sufficiently small such that $\mathbb{D}(z,
\varepsilon)%
\subseteq\Gamma$.
\end{pf}

\begin{lemma}\label{boundaryCondition}
If $\Gamma$ is a bounded domain, then
\[
\lim_{t \uparrow\tau_{\Gamma}}\Phi_{\Gamma}(B_{t}) =
\mathbf{1} \qquad\mbox{a.s. } P^{z}, z \in\Gamma.
\]
\end{lemma}

\begin{pf}
Fix $z \in\Gamma$, and let $B_{t}$ be the Brownian motion starting at
$z$-defined on $(\Omega, P^{z}, (\mathcal{F}_{t}), \mathcal{F})$. Let
$N_{t}: = \mathbb{E}^{z}[S(B_{[0, \tau_{\Gamma}]}) \vert\mathcal{F}_{t
\wedge\tau_{\Gamma}}]$. By Corollary~\ref{boundness}, the boundedness
of $\Gamma$ ensures that $S(B_{[0, \tau_{\Gamma}]})$ is
$L_{1}$-integrable, and it implies that $N_{t}$ is a uniformly
integrable martingale and by the martingale convergence theorem $\lim_{t \uparrow\infty} N_{t}$ exists a.s. and in $L^{1}$. Since $\Gamma$
is bounded, $\mathbb{E}[\tau_{\Gamma}] < \infty$, and thus $\tau_{\Gamma
}$ is finite a.s. in $P^{z}$. Therefore it follows that
\[
\lim_{t \uparrow\infty} N _{t} = \lim_{t \uparrow\tau_{\Gamma}}
\mathbb{E}^{z} \bigl[S(B_{[0, \tau_{\Gamma}]}) \vert\mathcal{F}_{t}
\bigr].
\]
Let $\{D_{k}\}$ be an increasing sequence of open sets such that $D_{k}
\subset\subset\Gamma$ and $\Gamma= \bigcup_{k} D_{k}$. Let $\tau_{k}:=
\tau_{D_{k}}$, and let $\mathcal{F}_{k}$ denote the $\sigma$-algebra
generated by the Brownian motion up to $\tau_{k}$. Since $\tau_{k}$ is
finite a.s., $\lim_{k \uparrow\infty}\tau_{k} = \tau_{\Gamma}$. Let
$\mathcal{F}_{\infty}$ be the $\sigma$-algebra generated by $\{\mathcal
{F}_{k}\}_{\tau_{k} \geq0}$. Since $S(B_{[0, \tau_{\Gamma}]})$ is
$L_{1}$-integrable, then $M_{k} := N _{\tau_{k}} $ is a discrete
martingale. By the martingale convergence theorem, it holds both a.s.
and in $L^{p}(P^{z})$ for every $p > 0$ that
%
\begin{equation}
\label{martingaleInequality1} \lim_{k \rightarrow\infty} M_{k} =
\mathbb{E}^{z} \bigl[S(B_{[0, \tau_{\Gamma
}]}) \vert\mathcal{F}_{\infty}
\bigr] = S(B_{[0, \tau_{\Gamma}]}).
\end{equation}
Since $\tau_{k} \uparrow\tau_{\Gamma}$ a.s., and $\lim_{t \uparrow
\infty} N _{t}$ exists, then $\lim_{t \uparrow\infty} N _{t}$ must
coincide with $\lim_{k\uparrow\infty} N _{\tau_{k}} = S(B_{[0, \tau
_{\Gamma}]})$, that is,
%
\begin{equation}
\label{eqn32} \lim_{t \uparrow\tau_{\Gamma}} \mathbb{E}^{z}
\bigl[S(B_{[0, \tau_{\Gamma
}]}) \vert\mathcal{F}_{t} \bigr] =
S(B_{[0, \tau_{\Gamma}]})\qquad \mbox{a.s.}
\end{equation}
Recall that $\Phi_{\Gamma}(z):= \mathbb{E}^{z}[S(B_{[0, \tau_{\Gamma
}]})]$. By the strong Markov property of Brownian motion, the
multiplicativity of the signatures and $L^{1}$ integrability of the
signature of the stopped Brownian motion, we have that
%
\begin{equation}
\label{eqn31} \mathbb{E}^{z} \bigl[S(B_{[0, \tau_{\Gamma}]}) \vert
\mathcal{F}_{t} \bigr] = S(B_{[0, t]}) \otimes
\Phi_{\Gamma}(B_{t})\qquad \mbox{if }t < \tau_{\Gamma}.
\end{equation}
%
Substituting (\ref{eqn31}) into (\ref{eqn32}), we obtain that
\[
\lim_{t \uparrow\tau_{\Gamma}} S(B_{[0, t]}) \otimes\Phi_{\Gamma
}(B_{t})
= S(B_{[0, \tau_{\Gamma}]}) \qquad\mbox{a.s. } P^{x}, x \in\Gamma.
\]
Due to the continuity of the signature with respect to time, it is
obvious that
%
\begin{equation}
\label{eqn13} \lim_{t \uparrow\tau_{\Gamma}} S(B_{[0, t]}) =
S(B_{[0, \tau_{\Gamma
}]})\qquad \mbox{a.s. } P^{x}, x \in\Gamma.
\end{equation}
Since $S(B_{[0, t]})$ is invertible and it is nothing else but the
signature of the Brownian motion running backward from time $t$ to $0$,
(\ref{eqn13}) implies that
%
\begin{equation}
\label{eqn14} \lim_{t \uparrow\tau_{\Gamma}} \bigl(S(B_{[0, t]})
\bigr)^{-1} = \bigl(S(B_{[0, \tau
_{\Gamma}]}) \bigr)^{-1}\qquad
\mbox{a.s. } P^{x}, x \in\Gamma.
\end{equation}
Combining (\ref{eqn13}) and (\ref{eqn14}), we obtain that
\begin{eqnarray*}
\lim_{t\uparrow\tau_{\Gamma}} \Phi_{\Gamma}(B_{t}) &=& \lim
_{t
\uparrow\tau_{\Gamma}} \bigl(S(B_{[0, t]}) \bigr)^{-1} \otimes
\lim_{t \uparrow
\tau_{\Gamma}} S(B_{[0, t]}) \otimes\Phi_{\Gamma}(B_{t})
\\
&=& \bigl(S(B_{[0, \tau_{\Gamma}]}) \bigr)^{-1} \otimes
S(B_{[0, \tau_{\Gamma}]})= \mathbf{1}.
\end{eqnarray*}
Now the proof is complete.
\end{pf}

%

%
\begin{theorem}[(A PDE for the expected signature of Brownian motion)]\label{PDE_theorem}
Assume $\Gamma$ is a bounded domain. Then $\Phi_{\Gamma}$ satisfies the
following PDE:
%
\begin{equation}
\qquad \bigtriangleup\Phi_{\Gamma}(z)=- \Biggl(\sum
_{i=1}^{d}e_{i}\otimes e_{i}
\Biggr)\otimes \Phi_{\Gamma}(z)-2\sum_{i=1}^{d}
\biggl( e_{i}\otimes\frac{\partial
\Phi
_{\Gamma}(z)}{\partial z_{i}} \biggr)\qquad \mbox{if }z\in
\Gamma,
\end{equation}
with the boundary condition that for every $z \in\Gamma$,
%
\begin{equation}
\lim_{t \uparrow\tau_{\Gamma}}\Phi_{\Gamma}(B_{t})=\mathbf{1}
\qquad \mbox{a.s. in } P^{z},
\end{equation}
and the initial condition:
%
\begin{eqnarray}
\rho_{0} \bigl(\Phi_{\Gamma}(z) \bigr) &=&1 \qquad\mbox{if }z
\in \bar{ \Gamma},
\\
\rho_{1} \bigl(\Phi_{\Gamma}(z) \bigr) &=&0 \qquad\mbox{if }z
\in \bar{ \Gamma}.
\end{eqnarray}
\end{theorem}

\begin{pf}
For any fixed $z\in\Gamma$, let us consider a function
\begin{eqnarray*}
\varphi\dvtx \bar{\varepsilon} \mathbb{D} &\rightarrow&T ( ( E ) ),
\\
y &\mapsto&\Psi(y)\otimes\Phi_{\Gamma}(z+y),
\end{eqnarray*}
where $\bar{\varepsilon} =\operatorname{dist}(z,\partial\Gamma)$.

By Lemma~\ref{Mean_value_property_lemma}, for every $\varepsilon< \bar{
\varepsilon}$,
\[
\Phi_{\Gamma}(z)=\frac{1}{d \omega_{d}\varepsilon^{d-1}}\int_{\partial
{%
\mathbb{D}(0, \varepsilon)}}\Psi(y)
\otimes\Phi_{\Gamma} (z+y)\,d\sigma(y)
\]
for every $z\in\Gamma$.

This implies that $\varphi$ satisfies the mean value property at $0$,
that is,
%
\begin{equation}
\varphi(0)=\Psi(0)\otimes\Phi_{\Gamma} (z)=\Phi_{\Gamma} (z)=
\frac{1}{d
\omega_{d}\varepsilon^{d-1}}\int_{\partial{\mathbb{D}(0, \varepsilon)}%
}\varphi(y)\,d\sigma(y),
\end{equation}
for every $\varepsilon\leq\bar{\varepsilon}$.

Since $\Gamma$ is a bounded domain, and in Lemma~\ref
{Lemma_Psi_smooth} we prove that $\Phi_{\Gamma}$ is a well defined
and twice differentiable function in the
componentwise sense, we see that so is $\varphi$. By the mean value
property and
differentiability of $\varphi$ at point $0$ we immediately have that
\[
\Delta \bigl(\varphi(y) \bigr)|_{y=0}=0.
\]
By the chain rule, we obtain
%
\begin{eqnarray}
\label{laplacian} &&\Delta \bigl(\varphi(y) \bigr)|_{y=0}
\nonumber
\\
&&\qquad=\Delta \bigl(\Psi(0) \bigr)\otimes\Phi_{\Gamma}(z)+2\sum
_{i=1}^{d}\frac
{\partial
(\Psi(0))}{\partial z_{i}}\otimes
\frac{\partial(\Phi(z))}{\partial
z_{i}}%
+\Psi(0)\otimes\Delta \bigl(\Phi_{\Gamma}(z)
\bigr)
\\
&&\qquad=0.
\nonumber
\end{eqnarray}
Recalling Lemma~\ref{Psi_remark}, we have the following equalities:
%
\begin{eqnarray}
\Psi(0) &=&\mathbf{1} \label{Psi_eqn},
\\
\frac{\partial(\Psi(0))}{\partial z_{i}} &=&e_{i}\qquad\mbox{for }i=1, 2, \ldots, d\quad \mbox{and}
\\
\Delta \bigl(\Psi(0) \bigr) &=&\sum_{i=1}^{d}e_{i}
\otimes e_{i}.
\end{eqnarray}
Thus after substituting these into (\ref{laplacian}) and rearranging
the equation, we finally obtain
%
\begin{equation}
\bigtriangleup\Phi_{\Gamma}(z)=- \Biggl(\sum
_{i=1}^{d}e_{i}\otimes e_{i}
\Biggr)\otimes \Phi_{\Gamma}(z)-2\sum_{i=1}^{d}
\biggl( e_{i}\otimes\frac{\partial
\Phi
_{\Gamma}(z)}{\partial z_{i}} \biggr).
\end{equation}
The boundary condition is proved in Lemma~\ref{boundaryCondition}.
Moreover by the definition of the signature, it is obvious that $\rho
_{0}(\Phi(z))=1$
and
\[
\rho_{1} \bigl(\Phi(z) \bigr)=\sum_{i=1}^{d}
\mathbb{E}^{z} \biggl[\int_{0}^{\tau
_{\Gamma}}\circ
dB_{t}^{(i)} \biggr]e_{i}=\sum
_{i=1}^{d}\mathbb{E}%
^{z}
\bigl[B_{\tau_{\Gamma}}^{(i)}-z^{(i)} \bigr]e_{i}=0,
\]
since Brownian motion is a martingale.
\end{pf}

\begin{remark}
In the proof of Theorem~\ref{PDE_theorem}, the essential condition for
the domain $\Gamma$ is actually the
well definedness and the componentwise differentiability of the
function $\Phi_{\Gamma}$, for which the boundedness of the domain is a
sufficient, but not a necessary condition.
\end{remark}

The following corollary is an equivalent statement of Theorem~\ref
{PDE_theorem}, which states how to use the PDE system to solve each
term of the
expected signature recursively.

\begin{corollary}
\label{PDE_corollary} Let $\Gamma$ be a bounded domain. For every $n\in
\mathbb{N}$ and $n\geq2$, the $n$th term of $\Phi_{\Gamma}$ satisfies
the following PDE:
%
\begin{eqnarray}
\label{PDE}&& \Delta \bigl(\rho_{n} \bigl(\Phi_{\Gamma}(z)
\bigr) \bigr)
\nonumber
\\[-8pt]
\\[-8pt]
\nonumber
&&\qquad=-2\sum_{i=1}^{d}e_{i}
\otimes \frac{%
\partial\rho_{n-1}(\Phi_{\Gamma}(z))}{\partial z_{i}}- \Biggl( \sum_{i=1}^{d}e_{i}
\otimes e_{i} \Biggr) \otimes\rho_{n-2} \bigl(\Phi
_{\Gamma
}(z) \bigr),
\end{eqnarray}
with the boundary condition that for each $z \in\partial\Gamma$,
\begin{eqnarray*}
\lim_{t\uparrow\tau_{\Gamma}}\rho_{n} \bigl(\Phi_{\Gamma}(B_{t})
\bigr)=\cases{ %
 0, & \quad$\mbox{if $n \geq1$;}$ \vspace*{2pt}
\cr
1, &
\quad $\mbox{if $n = 0$.}$ }
\end{eqnarray*}
Moreover $\rho_{0}(\Phi_{\Gamma}(z))=1,\rho_{1}(\Phi_{\Gamma
}(z))=0,\forall z\in\bar{\Gamma}$.
\end{corollary}

\begin{remark}
From the corollary it is important to notice that if we have computed
the $(n-1)$th and $(n-2)$th term of expected signature, the right-hand
side of the PDE of the $n$th term is known. There is a stochastic
representation to the solution to this generalized Poisson equation
problem with the appropriate Dirichlet boundary condition (Lemma~\ref
{boundaryCondition}), and as the solutions are all bounded, this
implies that from the PDE and the boundary condition we can resolve all
the terms of expected signature recursively \cite{oksendal2003stochastic}. The result is summarized in the following
theorem. 
\end{remark}

%
%
%

\begin{theorem}
Let $\Gamma$ be a bounded domain. For each $n\in\mathbb{N}$,
\[
\varphi_{n}(z):=- \Biggl( \sum_{i=1}^{d}e_{i}
\otimes e_{i} \Biggr) \otimes\rho _{n-2} \bigl(
\Phi_{\Gamma}(z) \bigr)-2 \Biggl( \sum_{i=1}^{d}e_{i}
\otimes\frac
{\partial
\rho_{n-1}(\Phi_{\Gamma}(z))}{\partial z_{i}} \Biggr).
\]
Suppose for each $n \in\mathbb{N}$,
\[
\mathbb{E}^{z} \biggl[\int_{0}^{\tau_{\Gamma}}
\varphi_{n}(B_{t})\,dt \biggr]<\infty\qquad\forall z\in
\Gamma.
\]
%
Then it follows that
\[
\rho_{n} \bigl(\Phi_{\Gamma}(z) \bigr)=\mathbb{E}^{z}
\biggl[\int_{0}^{\tau_{\Gamma
}}%
\varphi_{n}(B_{t})\,dt \biggr].
\]
\end{theorem}

\subsection{A concrete example: Brownian motion in the unit ball}
Recall that $E = \mathbb{R}^{d}$, and $\mathbb{D}$ denotes the open
unit ball in $E$ centered at the origin. In this subsection, we will
discuss the expected signature of $d$-dimensional Brownian motion
starting at $z \in\mathbb{D}$ upon the first exit time from $\mathbb
{D}$, denoted by $\Phi_{\mathbb{D}}(z)$ as before.

Let us start with the one-dimensional case, where $\Phi_{\mathbb{D}}$
can be solved explicitly. For $d=1$, $\mathbb{D}$ is just the interval
$(-1,1)$. After an easy computation, the probability of hitting $1$ at
the exit time is
\[
\mathbb{P}^{x}(B_{\tau_{(-1,1)}}=1)=\frac{x+1}{2},
\]
and it follows immediately that
\begin{eqnarray*}
\rho_{n}(\Phi_{(-1,1)}) &=&\frac{(1-x)^{n}}{n!}
\mathbb{P}%
^{x}(B_{\tau_{(-1,1)}}=1)+\frac{(-1-x)^{n}}{n!}
\mathbb{P}%
^{x}(B_{\tau_{( -1,1)}}=-1)
\\
&=&\frac{1}{2n!} \bigl(1-x^{2} \bigr) \bigl( (1-x)^{n-1}-(-1-x)^{n-1}
\bigr).
\end{eqnarray*}
It is easy to verify that $\Phi_{(-1,1)}(x)$ satisfies the following
ODE, which is consistent with Corollary~\ref{PDE_corollary}:
\[
\frac{d^{2}(\rho_{n}(\Phi_{( -1,1)}(x)))}{dx^{2}}=-\rho_{n-1} \bigl(\Phi _{(-1,1)}(x) \bigr)-2
\frac{d(\rho_{n-2}(\Phi_{(-1,1)}(x)))}{dx}\qquad \forall n\geq2,
\]
with the boundary condition
\[
\rho_{n} \bigl(\Phi_{(-1,1)}(x) \bigr)=0\qquad\mbox{if }x=\pm1
\ \forall n\geq2
\]
and the initial condition
\begin{eqnarray*}
\rho_{0} \bigl(\Phi_{(-1,1)}(x) \bigr) &=&1\quad \mbox{and}
\\
\rho_{1} \bigl(\Phi_{(-1,1)}(x) \bigr) &=&0\qquad \forall x
\in[-1,1].
\end{eqnarray*}
%

After computing $\Phi_{(-1, 1)}$, we are going to show that the
expected signature of the $d$-dimensional
Brownian motion upon the first exit time from the unit ball is in
polynomial form using Corollary~\ref{PDE_corollary}.
To be precise, we introduce the definition of a polynomial form for
a function mapping from a domain $\Gamma\subset E$ to $E^{\otimes n}$.

\begin{definition}
Let $f_{n}$ be a map from $\Gamma$ to $E^{\otimes n}$, where $\Gamma$
is a domain in $E$ endowed
with the canonical basis. Let $g$ be a polynomial in $E$. We say that
$f_{n}$ is in polynomial form with
a factor $g$ if for every index $I$ with length $n$, $\pi^{I}\circ
f_{n}$ is~a~polynomial with a common factor $g$. The degree
of the polynomial form $f_n$ is defined as the maximum of degrees of
all $\pi^{I} \circ f_{n}$ over all the indexes $I$ of length~$n$.
\end{definition}

To show that for every $n \in\mathbb{N}$, $\rho_{n} \circ\Phi_{\mathbb
{D}}$ is in polynomial form, which is summarized in Theorem~\ref
{ES_unit_disk_poly}, we need the following auxiliary lemma.

\begin{lemma}
\label{polynomial} For any given polynomial of degree $n$ denoted by $f\dvtx \mathbb{D} \rightarrow\mathbb{R}$, the solution to the PDE
\[
\cases{ %
\Delta F(z) = f(z), &\quad $\mbox{if $z \in \mathbb{D}$;}$
\vspace*{2pt}
\cr
F(z) = 0, &\quad $\mbox{if $|z| = 1$}$}
\]
exists and it is unique. Moreover it is a polynomial of degree $n+2$
with a
factor of $(1-\sum^{d}_{i=1}z_{i}^{2})$.
\end{lemma}

\begin{pf}
For any fixed $n \in\mathbb{N}$, let us denote the space of the polynomials
with degree no more than $n$ by $P[n]$. Define the linear operator $L$ which
maps $P[n]$ to $P[n]$ as follows:
\[
L(f) (z) := \Delta \Biggl( \Biggl(1-\sum^{d}
_{i=1}z_{i}^{2} \Biggr)f(z) \Biggr),
\]
where $f \in P[n]$.

We are going to show that $\operatorname{dim}(\operatorname{Im}(L))=\operatorname{dim}(P[n])$. Let us consider
$\operatorname{Ker}(L) =
\{ g \in P[n]| L(g) = 0 \}$. For every $g \in \operatorname{Ker}(L)$, let $G$ be
$G(z):=(1-%
\sum^{d}_{i=1}z_{i}^{2})g(z)$. Then $G \in P[n+2]$ and
satisfies the PDE problem
\[
\Delta G(z) = 0\qquad \mbox{if }z \in\mathbb{D},
\]
and the boundary condition
\[
G(z) = 0\qquad \mbox{if } |z| = 1.
\]
By the strong maximum principle $G(z)=0$ is the unique solution to this PDE
problem, so we have $\operatorname{Ker}(L)=\{0\}$ and $\operatorname{dim}(\operatorname{Ker}(L))= 0$. Since $L$ is a
linear operator, by rank-nullity theorem, it is easy to see that $\operatorname{Im}(L)
\subseteq P[n]$. Then we can claim that $L$ is a bijection and $\operatorname{Im}(L) =
P[n]$, which means that for every $f \in P[n]$, there exists a unique
polynomial $g \in P[n]$ such that
\[
L(g) = f.
\]
Equivalently for every $f \in P[n]$, there exits a unique polynomial
$F\in P[n+2]$ such that $F$ is the unique solution to the following PDE:
\[
\cases{ %
\Delta F(z) = f(z), &\quad $\mbox{if $z \in \mathbb{D}$;}$
\vspace*{2pt}
\cr
F(z) = 0, &\quad $\mbox{if $|z| = 1$,}$}
\]
and moreover $F(z)$ has a factor $(1-\sum^{d}_{i=1}z_{i}^{2})$.
\end{pf}

The following proof is an alternative way to prove Lemma~\ref
{polynomial} restricted for the two-dimensional case, but it gives us
an algorithm to compute $F$ explicitly.

\begin{pf*}{Proof of Lemma \ref{polynomial}}
We adopt the notation of $P[n]$ and the linear operator $L$ used in the
above proof. To prove the lemma, it is equivalent to prove that for
$f\in
P[n]$ with $n\geq0$, there exists a homogeneous polynomial denoted by $
g_{n} $ and a polynomial $R_{n}$ which is a polynomial of degree strictly
less than $n$ if $n\geq1$ and $R_{0}=0$ such that
\[
L(g_{n})=f^{\ast}+R_{n},
\]
where $f^{\ast}$ is the leading term of $f$.

For $n=0$, each polynomial $f\in P[n]$ is a constant function, that is,
\mbox{$f(z)=c$} $\forall z=(z_{1},z_{2})\in\mathbb{D}$. Then
$g(z):=-\frac{c}{4}$
satisfies
%
\begin{equation}
\Delta \Biggl( \Biggl(1-\sum^{2}
_{i=1}z_{i}^{2} \Biggr)g(z) \Biggr)=f(z)=c.
\end{equation}
It is easy to show that the solution to this PDE is unique. Thus the
statement is true for $n=0$.

For $n\geq1$, since $g_{n}$ is a homogeneous polynomial of degree $n$, we
have
\[
L(g_{n}) (z)=-4(n+1)g_{n}(z)+ \Biggl(1-\sum
_{i=1}^{2}z_{i}^{2} \Biggr)
\Delta \bigl(g_{n}(z) \bigr).
\]
The leading term of $L(g_{n})(z)$ is
\[
-4(n+1)g_{n}(z)-\sum_{i=1}^{2}z_{i}^{2}
\Delta \bigl(g_{n}(z) \bigr),
\]
which should match $f^{\ast}$.

Let us write $g_{n}$ in the form
\[
g_{n}(z)=\sum_{j=0}^{n}a_{j}z_{1}^{j}z_{2}^{n-j}
\]
and $f^{\ast}$ in the form
\[
f^{\ast}(z)=\sum_{j=0}^{n}b_{j}z_{1}^{j}z_{2}^{n-j}.
\]
Then
\begin{eqnarray*}
&&-4(n+1)g_{n}(z)-\sum_{i=1}^{2}z_{i}^{2}
\Delta \bigl(g_{n}(z) \bigr)
\\
&&\qquad=\sum_{j=0}^{n}
\bigl[-4(n+1)-j(j-1)1_{\{2\leq j\leq n\}}\\
&&\hspace*{48pt}{}-(n-j) (n-j-1)1_{\{
0\leq
j\leq n-2\}}
\bigr]a_{j}z_{1}^{j}z_{2}^{n-j}
\\
&&\qquad\quad-\sum_{j=0}^{n-2}(j+2)
(j+1)a_{j+2}z_{1}^{j}z_{2}^{n-j}\\
&&\qquad\quad{}-%
\sum_{j=2}^{n}(n-j+2) (n-j+1)a_{j-2}z_{1}^{j}z_{2}^{n-j}.
\end{eqnarray*}
Whether there exists $g_{n}$ such that the leading term of $\Delta((1-%
\sum^{2}_{i=1}z_{i}^{2})g_{n}(z))$ is $f^{\ast}(z)$
depends on whether the following linear equation has solution or not:
\begin{eqnarray*}
&& \bigl[-4(n+1)-j(j-1)1_{\{2\leq j\leq n\}}-(n-j) (n-j-1)1_{\{0\leq j\leq
n-2\}}
\bigr]a_{j}
\\
&&\quad{}-(j+2) (j+1)a_{j+2}1_{\{0\leq j\leq n-2\}}-a_{j-2}(n-j+2)
(n-j+1)1_{\{
2\leq
j\leq n\}}\\
&&\qquad=b_{j}
\end{eqnarray*}
for $j=0,1,\ldots,n$, or equivalently in the matrix form
%
\begin{equation}
M_{n}\vec{a}=\vec{b}, \label{linear_equation}
\end{equation}
where $\fontsize{8.36}{10.36}{\selectfont\vec{a}=\pmatrix{
a_{0} \cr
a_{1} \cr
\dvtx \cr
a_{n}} }$, $\fontsize{8.36}{10.36}{\selectfont\vec{b}=\pmatrix{
b_{0} \cr
b_{1} \cr
\dvtx \cr
b_{n}}}$, and $M_{n}$ is a $(n+1)\times(n+1)$ matrix which is defined as
below:
\[
\cases{ %
M_{n}(j,j)= \bigl[-4(n+1)-(n-j) (n-j-1) \bigr], \vspace*{2pt}\cr
\hspace*{63pt}\qquad\mbox{if $j < 2$;} \vspace*{2pt}
\cr
M_{n}(j,j)=
\bigl[-4(n+1)-j(j-1)-(n-j) (n-j-1) \bigr], \vspace*{2pt}\cr
\hspace*{63pt}\qquad \mbox{if $j \geq2, j \leq
n-2$;} \vspace*{2pt}
\cr
M_{n}(j,j)= \bigl[-4(n+1)-j(j-1) \bigr], \vspace*{2pt}\cr
\hspace*{63pt}\qquad \mbox{if $j > n-2, j \leq n$;} \vspace*{2pt}
\cr
M_{n}(j,j+2)=-(j+2)
(j+1), \vspace*{2pt}\cr
\hspace*{63pt}\qquad\mbox{if $ j \leq n-2$;} \vspace*{2pt}
\cr
M_{n}(j,j-2)=-(n-j+2)
(n-j+1), \vspace*{2pt}\cr
\hspace*{63pt}\qquad \mbox{if $j \geq2$;} \vspace*{2pt}
\cr
M_{n}(i,j)=0,
\qquad \mbox{otherwise.}}
\]
The final step is to prove that $M_{n}$ is invertible. It is easily verified
that $M_{n}$ is strictly diagonally dominant, since we have
\[
\bigl|M_{n}(j,j) \bigr|-\mathop{\sum_{i=0}}_{i\neq j}^{n}
\bigl|M_{n}(i,j) \bigr|=4(n+1)>0.
\]
By the Gershgorin's theorem \cite{GershgorinThm}, a strictly
diagonally dominant matrix is nonsingular. So $M_{n}$ is invertible, and
there exists a unique solution $\vec{a}$ to equation (\ref{linear_equation}),
\[
\vec{a}=M_{n}^{-1}\vec{b}
\]
that is, there exists a unique $g_{n}$ such that the leading term of
$\Delta((1-\sum^{2}_{i=1}z_{i}^{2})g_{n}(z))$
equals to $f^{\ast}$%
. By induction, we conclude that for every polynomial $f$ of degree $n$,
there exists a unique polynomial $g$ of degree $n$ such that
\[
L(g) (z)=f(z).
\]
\upqed\end{pf*}


\begin{theorem}\label{ES_unit_disk_poly}
For each $n\in\mathbb{N}^{+}$, $\rho_{n} \circ\Phi_{\mathbb{D}}(z)$
is in polynomial
form of degree no more than n with a factor $(1-\sum^{d}_
{i=1}z_{i}^{2})$.
\end{theorem}

\begin{pf}
We prove this by induction.

For $n=1$, $\rho_{1}(\Phi_{\mathbb{D}}(z))=0$ and trivially has a
factor of $(1-\sum^{d}_{i=1}z_{i}^{2})$.

For $n=2$, by Corollary~\ref{PDE_corollary} we have that $\rho_{2}(\Phi
_{\mathbb{D}} (z))$
satisfies the following PDE:
%
\begin{equation}
\label{Phi_2_PDE} \cases{ %
\displaystyle \Delta\rho_{2} \bigl(
\Phi_{\mathbb{D}} (z) \bigr)=-\sum^{d}
_{i=1}e_{i} \otimes e_{i}, & \quad $\mbox{if $z
\in \mathbb{D}$;}$ \vspace*{2pt}
\cr
\rho_{2} \bigl(\Phi_{\mathbb{D}}
(z) \bigr)=0, &\quad $\mbox{if $|z|=1$.}$}
\end{equation}
By Lemma~\ref{polynomial},
\[
\rho_{2} \bigl(\Phi_{\mathbb{D}}(z) \bigr)=\sum
_{i=1}^{d}\frac{1}{2d} \Biggl(1-\sum
^ {d}_{j=1}z_{j}^{2}
\Biggr)e_{i} \otimes e_{i}
\]
is the unique solution to (\ref{Phi_2_PDE}). So our statement is true
for $n=1,2$.

Suppose that the statement is true for every $n<N$. We are going to
prove that it is
true for $n=N$.

By Corollary~\ref{PDE_corollary}, we obtain that for every $n\geq2$,
\[
\bigtriangleup\rho_{n} \bigl(\Phi_{\mathbb{D}}(z) \bigr)=- \Biggl(
\sum_{i=1}^{d}e_{i}\otimes
e_{i} \Biggr) \otimes\rho_{n-2} \bigl(\Phi_{\mathbb{D}}(z)
\bigr)-2\sum_{i=1}^{d} \biggl(
e_{i}\otimes \frac{\partial\rho_{n-1}(\Phi_{\mathbb{D}}(z))}{\partial z_{i}} \biggr).
\]
By the induction hypothesis, it is easy to show that the right-hand
side should be in polynomial form of degree no more than $n-2$. Using
Lemma~\ref{polynomial}, each $\pi^{I}\circ\Phi_{\mathbb{D}}$ with
$|I|=n$ should be a polynomial of degree no more than $n$ with a factor
of $1-\sum^{d}_{i=1}z_{i}^{2}$, so $\rho_{n}(\Phi
_{\mathbb{D}})$ is in polynomial form of
degree no more than $n$ with a common factor $1-\sum^{d}_
{i=1}z_{i}^{2}$. Now our proof is complete.
\end{pf}

%
\begin{remark}
It is natural to guess that $\Phi_{\mathbb{D}}(z)$ has a common factor
$1- \vert z\vert^{2}$, since it would automatically satisfy the
boundary condition that
\[
\Phi(z) = \mathbf{1}\qquad \forall\vert z \vert= 1.
\]
\end{remark}

In the last part of this section, we give the following truncated
expected signature of two-dimensional Brownian motion upon the
first exit time from the unit disk up to degree $4$:
\begin{eqnarray*}
\rho_{2} \bigl(\Phi_{\mathbb{D}}(z) \bigr)&=& \frac{1}{4}
\Biggl(1- \sum_{i=1}^{2}z_{i}^{2}
\Biggr) \Biggl(\sum_{i=1}^{2}e_{i}
\otimes e_{i} \Biggr),
\\
\rho_{3} \bigl(\Phi_{\mathbb{D}}(z) \bigr) &=& \Biggl(1- \sum
_{i=1}^{2}z_{i}^{2}
\Biggr) \Biggl(-\sum_{i=1}^{2}
\frac{1}{8%
}z_{i}e_{i} \Biggr)\otimes \Biggl(\sum
_{i=1}^{2}e_{i}\otimes
e_{i} \Biggr) \qquad\mbox{and}
\\
\rho_{4} \bigl(\Phi_{\mathbb{D}}(z) \bigr)&=& \Biggl(1-\sum
_{i=1}^{2}z_{i}^{2}
\Biggr) \\
&&{}\times\biggl(D_{1}(z)e_{1}\otimes e_{1}+
\frac{z_{1}z_{2}}{24}(e_{1}\otimes e_{2} + e_{2}
\otimes e_{1})+D_{2}(z)e_{2}\otimes e_{2}
\biggr)
\\
&&{} \otimes \Biggl( \sum_{i=1}^{2}
e_{i}\otimes e_{i} \Biggr),
\end{eqnarray*}
where
\begin{eqnarray*}
D_{1}(z) &=& \tfrac{7}{192}z_{1}^{2}-
\tfrac{1}{192}z_{2}^{2}+\tfrac
{1}{64},
\\
D_{2}(z) &=& \tfrac{7}{192}z_{2}^{2}-
\tfrac{1}{192}z_{1}^{2}+\tfrac{1}{64}.
\end{eqnarray*}
%
\subsection{The geometric bounds for \texorpdfstring{$\Phi_{\Gamma}$}{$Phi_{Gamma}$}}\label{GeoBoundsForPhi}
In this subsection, we aim to show that under certain smoothness
condition of a bounded domain $\Gamma$, each term of $\Phi_{\Gamma}$ is
geometrically bounded. In order to do so, we start with estimating the
upper bounds for $W^{m, 2}$ norm of $\rho_{n}\circ\Phi_{\Gamma}$,
using our PDE theorem and then show that $\rho_{n}(\Phi_{\Gamma})(z)$
is geometrically bounded by the Sobolev theorem.

\subsubsection{$W^{m,2}$ bounds for \texorpdfstring{$\rho_{n} \circ\Phi_{\Gamma}$}{$rho_{n}circ Phi_{Gamma}$}}

\begin{lemma}\label{lemma_pde}
Let $\Gamma$ be a bounded domain of class $C^{m}$ in $E = \mathbb
{R}^{d}$, where $m = \lfloor\frac{d}{2}\rfloor+1$. Then there exists a
constant $C$ only depending on $\Gamma$ and $d$, such that for every
positive integer $n \geq2$,
%
\begin{equation}
\label{inequality1} \bigl\Vert\rho_{n}(\Phi_{\Gamma})
\bigr\Vert_{W^{m,2}
(\Gamma)} \leq C \bigl( \bigl\Vert\rho_{n-1}(
\Phi_{\Gamma}) \bigr\Vert_{W^{m,2}(\Gamma)} + \bigl\Vert\rho_{n-2}(
\Phi_{\Gamma}) \bigr\Vert_{W^{m,2}(\Gamma)} \bigr).
\end{equation}
\end{lemma}

\begin{pf}
Since $\Gamma$ is a bounded domain of class $C^{m}$, according to
Theorem~\ref{boundaryRegularity} there exists a constant $C_{1}$
depending only on $\Gamma$ and $d$, independent of $\rho_{n}(\Phi
_{\Gamma})$, such that
\[
\bigl\Vert\rho_{n}(\Phi_{\Gamma}) \bigr\Vert_{W^{m,2}(\Gamma)} \leq
C_{1} \bigl\Vert\Delta\rho_{n}(\Phi_{\Gamma}) \bigr\Vert
_{W^{m-2,2}(\Gamma)}.
\]
Recall the PDE of $\rho_{n}(\Phi_{\Gamma})$, that is,
%
\begin{eqnarray}
&&\Delta \bigl(\rho_{n} \bigl(\Phi_{\Gamma}(z) \bigr) \bigr)
\nonumber
\\[-8pt]
\\[-8pt]
\nonumber
&&\qquad=-2
\sum_{i=1}^{d}e_{i}\otimes
\frac{\partial\rho_{n-1}(\Phi_{\Gamma}(z))}{\partial z_{i}}- \Biggl(\sum_{i=1}^{d}e_{i}
\otimes e_{i} \Biggr) \otimes\rho_{n-2} \bigl(\Phi
_{\Gamma}(z) \bigr).
\end{eqnarray}
Then it follows immediately that
\begin{eqnarray*}
&& \bigl\Vert\rho_{n}(\Phi_{\Gamma})\bigr \Vert_{W^{m,2}(\Gamma)}
\\
&&\qquad\leq C_{1} \Biggl(\sum_{i=1}^{d}
\biggl\Vert  \frac{%
\partial\rho_{n-1}(\Phi_{\Gamma})}{\partial z_{i}} \biggr\Vert
_{W^{m-2,2}(\Gamma)}+ d \bigl\Vert\rho_{n-2} \bigl(\Phi_{\Gamma
}(z) \bigr)
\bigr\Vert_{W^{m-2,2}(\Gamma)} \Biggr)
\\
&&\qquad\leq C_{1} \Biggl(\sum_{i=1}^{d}
\bigl\Vert  \rho_{n-1}(\Phi _{\Gamma}) \bigr\Vert _{W^{m-1,2}(\Gamma)}+ d \bigl\Vert \rho_{n-2} \bigl(
\Phi_{\Gamma}(z) \bigr)\bigr\Vert_{W^{m-2,2}(\Gamma)} \Biggr).
\end{eqnarray*}
Since for every $f \in W^{k,2}(\Gamma)$, $ \Vert f \Vert
_{W^{k-1,2}(\Gamma)} \leq \Vert f \Vert_{W^{k,2}(\Gamma)}$,
then we choose $C=C_{1}d$, and (\ref{inequality1}) follows.
\end{pf}

%
\begin{lemma}\label{SobNormEs}
Let $\Gamma$ be a bounded domain of class $C^{m}$ in $\mathbb{R}^{d}$,
where $m = \lfloor\frac{d}{2}\rfloor+1$. Then there exists a constant
$C$ only depending on $\Gamma$ and $d$, such that for every integer $n
\geq0$,
\[
\bigl\Vert\rho_{n}(\Phi_{\Gamma}) \bigr\Vert_{W^{m,2}(\Gamma)} \leq \vert
\Gamma\vert^{{1}/{2}} C^{n}.
\]
\end{lemma}

\begin{pf}
By Lemma~\ref{lemma_pde}, there exists a constant $C_{1}>0$ such that
\[
\bigl\Vert\rho_{n}(\Phi_{\Gamma}) \bigr\Vert_{W^{m,2}(\Gamma)} \leq
C_{1} \bigl( \bigl\Vert\rho_{n-1}(\Phi_{\Gamma})\bigr \Vert
_{W^{m,2}(\Gamma)} + \bigl\Vert\rho_{n-2}(\Phi_{\Gamma}) \bigr\Vert
_{W^{m,2}(\Gamma)} \bigr).
\]
We choose $C=C_{1}+1$. Let us prove this statement by induction on $n$.
If $n = 0$, $\rho_{0}(\Phi_{\Gamma}(z)) = 1$ where $z \in\Gamma$, thus
\[
\bigl\Vert\rho_{0}(\Phi_{\Gamma}) \bigr\Vert_{W^{m,2}(\Gamma)} = \vert
\Gamma\vert^{{1}/{2}}\leq\vert\Gamma\vert^{{1}/{2}}C^{0}.
\]
It is obvious that if $n=1$, $\rho_{1}(\Phi_{\Gamma}(z)) = 0$ where
$z \in\Gamma$, thus
\[
\bigl\Vert\rho_{1}(\Phi_{\Gamma}) \bigr\Vert_{W^{m,2}(\Gamma)} = 0 \leq
\vert\Gamma\vert^{{1}/{2}}C^{1}.
\]
By the induction hypothesis, we have that
\begin{eqnarray*}
\Vert\rho_{n}\circ\Phi_{\Gamma}\Vert_{W_{m,2}(\Gamma)} &\leq &
C_{1} \bigl( \vert\Gamma\vert^{{1}/{2}}C^{n-1} + \vert
\Gamma \vert^{{1}/{2}}C^{n-2} \bigr)
\\
&=& \vert\Gamma\vert^{{1}/{2}}C^{n-2}C_{1}(C+1)\leq
\vert\Gamma \vert^{{1}/{2}} C^{n},
\end{eqnarray*}
since $C = C_{1} + 1$ and $C_{1}(C+1) = C^{2}-1\leq C^{2}$.
\end{pf}

%
\subsubsection{The geometric bounds for \texorpdfstring{$\vert\rho_{n}(\Phi_{\Gamma}(x))\vert$}{$|rho_{n}(Phi_{Gamma}(x))|$}}

\begin{theorem}
Let $\Gamma$ be of the class $C^{m}$ and strong Lipschitz in $\mathbb
{R}^{d}$. Then there exists a constant $C$ depending on $d$ and $\Gamma
$, such that for every $x \in\Gamma$, for every positive integer $n$,
\[
\bigl\vert\rho_{n} \bigl(\Phi_{\Gamma}(x) \bigr) \bigr\vert\leq
C^{n}.
\]
\end{theorem}

\begin{pf}
By Lemma~\ref{SobNormEs}, there exists a constant $C_{1}$ such that for
every $n \in\mathbb{N}$,
\[
\bigl\Vert\rho_{n} \circ\Phi_{\Gamma}\bigr\Vert_{W_{m, 2}(\Gamma)} \leq
C_{1}^{n}.
\]
According to Theorem~\ref{SobEmbThm}, there is a constant
$C_{2}(d,\Gamma)$ such that
\[
\bigl\vert\rho_{n}\circ\Phi_{\Gamma}(x) \bigr\vert\leq C_{2}(d,
\Gamma)\Vert\rho_{n}\circ\Phi_{\Gamma} \Vert_{W^{m,2}(\Gamma)}.
\]
%
Let $C =C_{1}\max\{C_{2}, 1\}$. Then it obvious that for every $x
\in\Gamma$,
\[
\bigl\vert\rho_{n}\circ\Phi_{\Gamma}(x) \bigr\vert\leq
C_{2}dC_{1}^{n}\leq C^{n}.
\]
\upqed\end{pf}
%

\begin{remark}
According to Chevyrev \cite{Ilya}, when the expected signature is
compact-like, it determines the law of the signature. This result
directs us to study the decay rates of $\Phi_{\Gamma}$. So far our best
result is that $\Phi_{\Gamma}$ is geometrically bounded, and it
provides insufficient information for us to conclude whether $\Phi
_{\Gamma}$ is compact-like or not. Because the geometric boundedness of
a tensor series does not imply that this tensor series is compact-like.
It would be still unclear even in the simplest case that $\Gamma=
\mathbb{D}$ whether the expected signature of Brownian motion
characterizes the law of the signature of stopped Brownian motion. One
difficulty comes from the tail behavior of higher-order iterated
integrals of Brownian motion; for example, the L\'evy area of the
Brownian motion stopped at the exit time from the disk is shown to have
only exponential tail \cite{NiHaoThesis}. However, the story is not
simply about tail behavior. The interaction between iterated integrals
is clearly of great importance. We can see this by looking at the
one-dimensional Gaussian random variable $X$. The law of $X$ is
certainly determined by its moments. On the other hand, it is known
that the law of $Y :=X^{3}$ is not determined via its moments \cite{berg1988cube}. Nonetheless the joint distribution of $(X, Y)$ is
determined by its moments. This is because one can deduce from the
moments of $(X, Y)$ that the expected value of $(X^{3}-Y)^{2}$ is zero,
hence recovering the algebraic relation $X^{3} = Y$ almost surely.
\end{remark}

\subsection{A discrete analogy: The expected signature of a simple
random walk up to an exit time}
Let $X_{1},X_{2},\ldots $ be independent, identically distributed random
variables on a probability space $(\Omega,\mathcal{F},\mathbf{P})$ taking
values in the integer lattice $\mathbb{Z}^{d}$ with
\[
\mathbf{P}\{X_{j}=e\}=\frac{1}{2d},\qquad|e|=1.
\]
A simple random walk starting at $x\in\mathbb{Z}^{d}$ is a stochastic
process $S_{n}$
indexed by the nonnegative integers, with $S_{0}=x$. We denote by
$\Gamma$
a regular domain of the integer lattice and denote by $\tau_{\Gamma}$ the
first exit time from $\Gamma$; the definition of a regular domain can be
found in \cite{lawler2012intersections}. The process $S_{n}$ can be
viewed as a random
lattice path of step~$n$, so that its iterated integrals are well defined.
Moreover $\Phi_{\Gamma}(x)$ is defined as the expected signature of a simple
random walk starting at $x$ and ending at the first exit time from
$\Gamma$ ($%
\Gamma\subseteq\mathbb{Z}^{d}$ is a finite set). By the
multiplicative property of the
signature and the strong Markov property of the simple random walk, we
have the
following important equation: for every $x \in\Gamma$,
%
\begin{equation}
\Phi_{\Gamma}(x)=\sum_{|e_{j}|=1}
\frac{1}{2d}\exp(e_{j})\otimes\Phi _{\Gamma}(x+e_{j}),
\label{Eqn1}
\end{equation}
where for $j \in\{1, \ldots, d\}$, $e_{j}$ is the unit vector in
$\mathbb{Z}^{d}$ with $j$th component $1$, and for $j \in\{d+1, \ldots,
2d\}$, $e_{j}$ is the unit vector in $\mathbb{Z}^{d}$ with $(j-d)$th
component $-1$.

Rewriting equation (\ref{Eqn1}), we have
\[
\rho_{n} \bigl(\Phi_{\Gamma}(x) \bigr)=\sum
_{|e_{j}|=1}\frac{1}{2d}\sum_{i=0}^{n}
\frac{%
(e_{j})^{\otimes i}}{i!}\otimes\rho_{n-i} \bigl(\Phi_{\Gamma}(x+e_{j})
\bigr).
\]
This is equivalent to
%
\begin{equation}
-\Delta\rho_{n} \bigl(\Phi_{\Gamma}(x) \bigr)=\sum
_{|e_{j}|=1}\frac{1}{2d}%
\sum
_{i=1}^{n}\frac{(e_{j})^{\otimes i}}{i!}\otimes
\rho_{n-i} \bigl(\Phi _{\Gamma}(x+e_{j}) \bigr),
\label{Pde_n_term}
\end{equation}
where $\Delta$ denotes the discrete Laplace operator; that is, for any
function $f\dvtx \Gamma\rightarrow E^{\otimes n}$,
\[
\Delta f(x) = \frac{1}{2d}\sum_{i=1}^{2d}
\bigl(f(x+e_{i})-f(x) \bigr).
\]
It is also easy to verify that for every $x \in\Gamma^{c}$, the
following equation holds:
%
\begin{equation}
\Phi_{\Gamma}(x)=\mathbf{1}=(1,0,0,\ldots),
\end{equation}
where $\Gamma^{c}$ is the complement of $\Gamma$ in $\mathbb{Z}^{d}$.

We summarize our result in the following theorem.

\begin{theorem}
Let $\Gamma\subseteq\mathbb{Z}^{d}$ be a finite set. Then $\Phi
_{\Gamma}\dvtx \mathbb{Z}^{d}\rightarrow T((E))$ satisfies the following
conditions:
\begin{longlist}[(1)]
\item[(1)] $\forall x \in\Gamma^{c}, \Phi_{\Gamma}(x) = \mathbf{1}$;
\item[(2)] $\forall x\in\Gamma, \rho_{0}(\Phi_{\Gamma}(x))=1,\rho
_{1}(\Phi_{\Gamma}(x))=0$;
\item[(3)] $\forall n \geq2$, $\forall x\in\Gamma$,
\[
\Delta\rho_{n} \bigl(\Phi_{\Gamma}(x) \bigr)=-\sum
_{|e_{j}|=1}\frac{1}{2d}%
\sum
_{i=1}^{n}\frac{e_{j}^{\otimes i}}{i!}\otimes
\rho_{n-i} \bigl(\Phi _{\Gamma}(x+e_{j}) \bigr).
\]
\end{longlist}
\end{theorem}

\begin{remark}
Notice that the right-hand side of (\ref{Pde_n_term}) is determined
totally by the truncated expected signature up to degree $n-1$. This indicates
that we can solve $\Phi_{\Gamma}(x)$ recursively just as in the Brownian
motion case. Classical potential theory guarantees that we can
recursively solve a system of finite difference problems to obtain the
whole expected signature of a simple random walk up to the first exit
time.%
\end{remark}

\begin{theorem}
\label{discrete_PDE_thereom} Let $\Gamma\subseteq\mathbb{Z}^{d}$ be a
finite set, $%
F\dvtx \Gamma^{c} \rightarrow\mathbb{R},g\dvtx \Gamma\rightarrow\mathbb{R}$.
Then the unique function $f\dvtx \mathbb{Z}^{d}\rightarrow\mathbb{R}$ satisfying
\begin{eqnarray*}
&&\mathrm{(a)}\quad\Delta f(x)=-g(x),\qquad x\in\Gamma,
\\
&&\mathrm{(b)}\quad f(x)=F(x),\qquad x\in\Gamma^{c},
\end{eqnarray*}
is
\[
f(x)=\mathbb{E}^{x} \Biggl[F(S_{\tau_{\Gamma}})+\sum
_{j=0}^{\tau_{\Gamma}
-1}g(S_{j}) \Biggr].
\]
\end{theorem}

Immediately we have the following theorem in our setting.
%
\begin{theorem}
Let $\Gamma\subseteq\mathbb{Z}^{d}$ be a finite set. Then $\Phi
_{\Gamma}\dvtx \mathbb{Z}^{d}\rightarrow T((E))$ is given as follows:
\begin{longlist}[(1)]
\item[(1)] $\forall x \in\Gamma^{c}, \Phi_{\Gamma}(x) = \mathbf{1}$;
\item[(2)] $\forall x\in\Gamma, \rho_{0}(\Phi_{\Gamma}(x))=1,\rho
_{1}(\Phi_{\Gamma}(x))=0$;
\item[(3)] $\forall n\geq2$, $\forall x \in\Gamma$,
\[
\rho_{n} \bigl(\Phi_{\Gamma}(x) \bigr)=\mathbb{E}^{x}
\Biggl[\sum_{j=0}^{\tau
_{\Gamma}
-1}g_{n}(S_{j})
\Biggr],
\]
where the integration is understood in componentwise sense and
\[
g_{n}(x)=\sum_{|e_{j}|=1}\frac{1}{2d}\sum
_{i=1}^{n}\frac
{(e_{j})^{\otimes i}%
}{i!}\otimes
\rho_{n-i} \bigl(\Phi_{\Gamma}(x+e_{j}) \bigr).
\]
\end{longlist}
\end{theorem}

%

\section*{Acknowledgments}
We thank the referees and the Associate Editor for their
thorough review and highly appreciate their comments and
suggestions, which significantly contributed to improve
the quality of this paper. We would both like to thank Oxford-Man
Institute for their funding.



%


\printaddresses
\end{document}